\documentclass{amsart}

\usepackage{hyperref}
\usepackage{enumerate}
\usepackage{amssymb}
\usepackage{graphicx}
\usepackage{MnSymbol}

\usepackage{amsthm}

\title{A widely connected topological space made from diamond}
\author{Samuel M. Corson}

\theoremstyle{definition}\newtheorem{theorem}{Theorem}

\theoremstyle{definition}

\numberwithin{theorem}{section}
\theoremstyle{definition}
\theoremstyle{definition}\newtheorem{proposition}[theorem]{Proposition}
\theoremstyle{definition}\newtheorem{definition}[theorem]{Definition}
\theoremstyle{definition}
\theoremstyle{definition}
\theoremstyle{definition}\newtheorem{remark}[theorem]{Remark}
\theoremstyle{definition}
\theoremstyle{definition}\newtheorem{lemma}[theorem]{Lemma}
\theoremstyle{definition}
\theoremstyle{definition}
\theoremstyle{definition}\newtheorem{definitions}[theorem]{Definitions}
\theoremstyle{definition}
\theoremstyle{definition}
\theoremstyle{definition}\newtheorem{construction}[theorem]{Construction}

\newcommand{\B}{\mathbb{B}}
\newcommand{\Rel}{\operatorname{Rel}}
\newcommand{\Log}{\operatorname{Log}}
\newcommand{\X}{\mathcal{X}}
\newcommand{\Coor}{\operatorname{Coor}}
\newcommand{\Conn}{\operatorname{Conn}}
\newcommand{\Bo}{\mathcal{B}}
\newcommand{\dom}{\operatorname{dom}}
\newcommand{\ord}{\operatorname{ord}}
\newcommand{\Sequ}{\operatorname{Sequ}}

\begin{document}

\address{School of Mathematics, University of Bristol, Fry Building, Woodland Road, Bristol, BS8 1UG, United Kingdom.}

\email{sammyc973@gmail.com}
\thanks{This work was supported by the Severo Ochoa Programme for Centres of Excellence in R\&D SEV-20150554 and by the Heilbronn Institute for Mathematical Research, Bristol, UK.}

\begin{abstract} We give the construction of an infinite topological space with unusual properties.   The space is regular, separable, and connected, but removing any nonempty open set leaves the remainder of the space totally disconnected (in fact, totally separated).  The space is also strongly Choquet (in fact, satisfies an even stronger condition) and has a basis with nice properties.  The construction utilizes Jensen's diamond principle $\diamondsuit$.
\end{abstract}

\maketitle

\begin{section}{Introduction}

The goal of this paper is to produce an unusual space which, though connected, is only tenuously connected and is quite topologically ``thick''.  We assume some familiarity with topological terminology in this introduction, but many definitions will be provided in Section \ref{Top}.  We adopt the fairly standard convention that regular spaces are also Hausdorff.

A point $x$ in a connected space $X$ is a \emph{dispersion point} if the subspace $X \setminus \{x\}$ is totally disconnected \cite{StSe}.  The famous Knaster-Kuratowski Fan (also known as Cantor's Leaky Tent) is a subspace of the plane which is infinite and connected and contains a dispersion point \cite{KnKu}.  One could imagine a connected space to be quite tenuously connected if it contained many dispersion points.  Unfortunately a connected space with at least three points can have at most one dispersion point.

A reasonable alternative to having many dispersion points in a connected space $X$ is to have the removal of any \emph{neighborhood} $O$ of any point ensure that the subspace $X \setminus O$ is totally disconnected (such spaces are called \emph{widely connected} in the literature \cite{Sw}).  An easy example of such a space is an infinite set with the cofinite topology.  Unfortunately this example does not satisfy very impressive separation axioms; it is $T_1$ but not Hausdorff.  Far more interesting widely connected spaces are known to exist (e.g. \cite{Gru} \cite{Lip}, \cite{Sw}).

We produce a space which is widely connected (in fact satisfying the stronger conditions (3) and (5) below), regular, separable and has a basis with quite nice properties.  We use Jensen's $\diamondsuit$ principle (see Section \ref{modelth}) in the construction in addition to the standard Zermelo-Fraenkel axioms of set theory with the axiom of choice (ZFC).  As ZFC + $\diamondsuit$ is consistent if and only if ZFC is consistent \cite{Jen}, the existence of our space is consistent with ZFC.

\begin{theorem}(ZFC + $\diamondsuit$) \label{main}  There exists a topological space $(X, \tau)$, having a basis $\Bo$, which is

\begin{enumerate}

\item[(1)] regular;

\item[(2)] separable;

\item[(3)] connected;

\item[(4)] of cardinality $\aleph_1$;

\end{enumerate}

\noindent and

\begin{enumerate}

\item[(5)] for every nonempty $O \in \tau$ the subspace $X \setminus O$ is totally separated (hence totally disconnected);

\item[(6)] every nonempty $O \in \tau$ is uncountable;

\item[(7)] $|\mathcal{B}| = \aleph_1$;

\item[(8)] every countable cover of $X$ by elements of $\Bo$ has a finite subcover; and

\item[(9)] every nesting decreasing countable sequence $\{B_n\}_{n\in \omega}$ of basis elements has $\bigcap_{n\in \omega} B_n \neq \emptyset$ (hence $X$ is strongly Choquet).

\end{enumerate}
\end{theorem}

We note that this result is reasonably sharp.  A space satisfying (1)-(9) above cannot have any of the following additional properties:

\begin{itemize}

\item metrizable

\item second countable

\item Lindel\"of

\item there exists a nonempty open set with compact closure
\end{itemize}

\noindent (see Proposition \ref{badprop}).  One also cannot strengthen (5) by replacing it with

\begin{enumerate}
\item[(5')] for each nonempty open $O \subseteq X$ the subspace $X \setminus O$ is zero-dimensional
\end{enumerate}

\noindent (see Proposition \ref{notzerodim}).  Thus the space is topologically small (8) and narrow (2), but no space satisfying (1) - (9) can be so small as to be Lindel\"of, or so narrow as to be second countable.  The topology is reasonably fine (1), but no space satisfying (1) - (9) can be metrizable.  The space is tenuously connected ((3) and (5)).  Moreover the disconnectedness occasioned by removing any region is about as strong as allowable in light of (1).  Despite the tenuous connectedness of the space, it is topologically thick (9).

One also cannot replace (4) with ``of cardinality $\aleph_0$'' since regular connected spaces with at least two points must be of cardinality at least $\aleph_1$, but interestingly one can produce (from ZFC) regular separable connected spaces of cardinality exactly $\aleph_1$ \cite{CiWo}.  Any (strongly) Choquet space having at least two points and satisfying (1) and (3) is necessarily of cardinality $\geq 2^{\aleph_0}$ (see Proposition \ref{cardinalityChoquet}), and this together with (4) highlights the well known fact that $\diamondsuit$ implies $\aleph_1 = 2^{\aleph_0}$ (see Remark \ref{CH}).

It may be that one can strengthen (1) by replacing ``regular'' by ``completely regular'' or ``normal.''  A completely regular connected space with at least two points must have cardinality at least $2^{\aleph_0}$, but since $\diamondsuit$ impies $\aleph_1 = 2^{\aleph_0}$ this strengthening might be possible under our set theoretic assumptions.  Another potential strengthening would be to replace (8) with ``$X$ is countably compact.''  One could also try producing the space of Theorem \ref{main}, with all instances of $\aleph_1$ replaced with $2^{\aleph_0}$, from ZFC alone.

If one removes the conditions on the basis then some extra properties are allowed.  Separable metrizable spaces satisfying (3) and (5)  have been obtained classically \cite{Sw}.  More recently, widely connected Polish spaces have been produced \cite{Lip}.  From Martin's Axiom one has spaces which are completely regular, or countable and Hausdorff, and satisfy (3) and (5) (in fact, the connected subsets are precisely the cofinite ones), and from the stronger assumption $\aleph_1 = 2^{\aleph_0}$ such a space can be perfectly normal \cite{Gru}.

Our space is constructed by induction over $\aleph_1$.  At certain steps of the induction we begin the construction of a new basis element, making committments as to how the basis elements interact with each other.  New points are assigned to the increasingly-defined basis elements, but we will not have finished the construction of any basis element until the very end of the the induction.  Connectedness is assured by detecting potential disconnections (using $\diamondsuit$) and blocking them by determining that certain sequences of points will converge.  The ability to block the detected disconnections will be assured by model-theoretic techniques.

We provide some topological preliminaries in Section \ref{Top}, and some model-theoretic preliminaries in Section \ref{modelth}. The space is constructed in Section \ref{construction} and the various properties claimed in Theorem \ref{main} are verified in Section \ref{verification}.
\end{section}

\begin{section}{Topology preliminaries}\label{Top}

We will review some very basic topology terminology and prove some easy preparatory results.  We will let $\omega$ denote the set of natural numbers and adopt the standard convention that an ordinal number is the set of ordinal numbers which are strictly smaller than itself (e.g. $3 = \{0, 1, 2\}$ and $\omega + 1 = \{0, 1, \ldots, \omega\}$).  If $\alpha$ is a successor ordinal we let $\alpha - 1$ denote its predecessor.  Throughout this paper, when we have an enumeration of a set (i.e. an indexing of its elements by a collection of ordinals) we will not be generally requiring that the indexing is injective, although often the indexing will be so.

\begin{definitions}  Recall that a space $X$ is \emph{Hausdorff} if for distinct $x_0, x_1 \in X$ there exist disjoint open sets $U_0, U_1\subseteq X$ such that $x_i \in U_i$.  We say $X$ is \emph{regular} if it is Hausdorff and for any closed subset $C \subseteq X$ and $x\in X \setminus C$ there exist disjoint open sets $U_x$ and $U_C$ such that $x\in U_x$ and $C \subseteq U_C$.  A space is \emph{separable} if it includes a countable dense subset, and Lindel\"of if every open cover has a countable subcover.
\end{definitions}

\begin{definition}  If $X$ is a topological space and $\mathcal{B}$ is a subset of $X$, we say that $\mathcal{B}$ is a \emph{basis} for the topology on $X$ if all elements of $\Bo$ are open and each open set in $X$ is a (possibly empty) union of elements of $\mathcal{B}$.
\end{definition}

It is a standard exercise to show that if $\mathcal{B}$ is a collection of subsets of a set $X$ such that $\bigcup \mathcal{B} = X$, and for any $B_0, B_1 \in \mathcal{B}$ and $x\in B_0 \cap B_1$ there exists $B_2 \in \mathcal{B}$ such that $x\in B_2 \subseteq B_0 \cap B_1$, then the collection of unions of elements in $\mathcal{B}$ is a topology on $X$ with basis $\mathcal{B}$.  A space is \emph{second countable} provided there exists a countable basis for the topology.

We now give some relevant definitions with regard to (dis)connectedness.

\begin{definitions}  Let $X$ be a topological space.  A writing $X = V_0 \sqcup V_1$ of $X$ as a disjoint union of open sets $V_0, V_1 \subseteq X$ is a \emph{disconnection}.  We call a disconnection \emph{trivial} is at least one of the $V_i$ is empty.  The space $X$ is \emph{connected} if each disconnection of $X$ is trivial.  A \emph{component} of a space is a maximal connected subspace.  Components are necessarily closed.  We say $X$ is \emph{totally disconnected} if all components are of cardinality at most $1$, and $X$ is \emph{totally separated} if for any distinct $x, y\in X$ there exists a disconnection $X = V_0 \sqcup V_1$ such that $x \in V_0$ and $y \in V_1$.  Also, $X$ is \emph{zero-dimensional} if there exists a basis for $X$ consisting of clopen sets.  Clearly zero-dimensional implies totally separated implies totally disconnected.
\end{definitions}

Given a subset $Y$ of a topological space $X$ we'll use $\overline{Y}$ to denote the closure of $Y$ in $X$ and $\partial Y$ to denote the boundary $\overline{Y} \cap \overline{(X \setminus Y)}$ of $Y$.

\begin{lemma}\label{limitpoint}  Suppose $X$ is a connected topological space and $X_0, X_1 \subseteq X$ are nonempty disjoint open subsets such that $X_0 \cup X_1$ is dense in $X$.  Then there exists a point $x \in (\overline{X_0}\cup \overline{X_1}) \setminus (X_0 \cup X_1)$.
\end{lemma}

\begin{proof}  Since $X = \overline{X_0 \cup X_1} = \overline{X_0} \cup \overline{X_1}$ we must have $\overline{X_0}\cap \overline{X_1}\neq \emptyset$, else $X =  \overline{X_0} \sqcup \overline{X_1}$ is a nontrivial disconnection.  Since $X_0$ is open and $X_0 \cap X_1 = \emptyset$ we know that $X_0 \cap \overline{X_1} = \emptyset$, and similarly $X_1 \cap \overline{X_0} = \emptyset$.  Selecting $x\in \overline{X_0}\cap \overline{X_1}$ we therefore have $x\in  (\overline{X_0}\cup \overline{X_1}) \setminus (X_0 \cup X_1)$.
\end{proof}

Next we'll provide some justification for why the properties of the space claimed in Theorem \ref{main} are fairly sharp.

\begin{proposition}\label{badprop}  If a space $(X, \tau)$ with basis $\Bo$ satisfies the properties (1) - (9) of Theorem \ref{main} then it cannot satisfy any of the following additional properties:

\begin{enumerate}

\item[(a)] metrizable;

\item[(b)] second countable;

\item[(c)] Lindel\"of;

\item[(d)] compact;

\item[(e)] there exists a nonempty open set with compact closure.
\end{enumerate}
\end{proposition}

\begin{proof}  The fact that (a) $\wedge$ (2) $\Rightarrow$ (b), (b) $\Rightarrow$ (c), (c) $\wedge$ (8) $\Rightarrow$ (d), and (d) $\Rightarrow$ (e) is fairly routine.  We'll derive a contradiction from (e) $\wedge$ (1) $\wedge$ (3) $\wedge$ (4) $\wedge$ (5) (in fact we will not use (1) but rather the weaker Hausdorff condition).  Suppose that $X$ is Hausdorff, connected, satisfies (5), has at least two points, and $O \subseteq X$ is a nonempty open set such that $\overline{O}$ is compact.  We can assume without loss of generality that $\overline{O}$ is a proper subset of $X$, since $X$ is Hausdorff and has at least two points.  As $X$ is connected, we have $\partial \overline{O} \neq \emptyset$.  Fix $x\in O$.  As $\overline{O}$ is totally separated, there exists for each $y\in \partial \overline{O}$ a pair of open sets $(U_{x, y}, U_{y, x})$ in $X$ such that $U_{x, y} \cap U_{y, x} \subseteq X \setminus \overline{O}$ and $x\in U_{x, y}$ and $y\in U_{y, x}$.  As $\partial \overline{O}$ is compact, there exists a finite subset $\{U_{y_0, x}, \ldots, U_{y_k, x}\}$ such that $\bigcup_{i = 0}^{k} U_{y_i, x} \supseteq \partial \overline{O}$.  Letting $U_0 = O \cap \bigcap_{i = 0}^k U_{x, y_i}$ and $U_1 = \bigcup_{i = 0}^k U_{y_i, x}$ we obtain a nontrivial disconnection $X = U_0 \sqcup ((X \setminus \overline{O}) \cup U_1)$, contradicting (3).
\end{proof}

\begin{proposition}\label{notzerodim}  If $X$ is a connected space and $U_0$ and $U_1$ are nonempty disjoint open sets  then $X \setminus U_0$ cannot be zero-dimensional.
\end{proposition}

\begin{proof}  Suppose otherwise for contradiction.  Then we may write $U_1$ as a union $U_1 = \bigcup_{i\in I}O_i$ where each $O_i$ is clopen in $X \setminus U_0$.  Since $U_1 \neq \emptyset$ we select $i \in I$ for which $O_i \neq \emptyset$.  As $O_i \subseteq U_1$ is open in $X \setminus U_0$, it is also open in $U_1$, and therefore open in $X$.  As $\overline{O_i} \subseteq \overline{U_1} \subseteq X \setminus U_0$ and $O_i$ is closed in $X \setminus U_0$, we know $O_i$ is closed in $X$.  Thus $\emptyset \neq O_i \subseteq X\setminus U_0 \neq X$ is clopen, contradicting the connectedness of $X$.
\end{proof}

\begin{definitions} (see \cite[Chapter 8]{Ke}) Recall that a space is \emph{Baire} if any countable intersection of open dense sets is dense.  Given a space $X$, the \emph{strong Choquet game} is an infinitary game between a Player I and a Player II, a run of which we describe.  Player I chooses a point and neighborhood $(x_0, U_0)$, then Player II chooses an open neighborhood $x_0 \in V_0 \subseteq U_0$, Player I chooses a point and neighborhood $(x_1, U_1)$ such that $U_1 \subseteq V_1$, Player II chooses a neighborhood $x_1 \in V_1 \subseteq U_1$, etc.  Player I wins the run if $\bigcap_{n\in \omega} U_n = \emptyset$ and otherwise Player II wins.  The Choquet game is similar, but the players are only selecting nonempty open sets so that $U_0 \supseteq V_1 \supseteq U_1 \supseteq \cdots$, and again Player I wins if $\bigcap_{n\in \omega} U_n = \emptyset$, with Player II winning otherwise.  A space is (strongly) Choquet if Player II has a winning strategy in the (strong) Choquet game.  Strongly Choquet implies Choquet implies Baire.
\end{definitions}

\begin{lemma}\label{stronglyChoquetfrombasis}  If a space $X$ has a basis such that any nesting decreasing countable sequence $\{B_n\}_{n\in \omega}$ of basis elements has nonempty intersection, then $X$ is strongly Choquet.
\end{lemma}

\begin{proof}  Assume the hypotheses.  We define a winning strategy for Player II: given a move $(x_n, U_n)$ we let $V_n$ be any basis element such that $x_n \in V_n \subseteq U_n$.  By assumption we get $\bigcap_{n\in \omega} V_n \neq \emptyset$, so this is a winning strategy.
\end{proof}

\begin{proposition}\label{cardinalityChoquet}  If $X$ is a nonempty, Hausdorff, Choquet space with no isolated points then $|X| \geq 2^{\aleph_0}$.  In particular the conclusion applies if $X$ is a regular, connected, strongly Choquet space with at least two points.
\end{proposition}

\begin{proof}  Obviously the second sentence of the proposition follows from the first, since a connected space with at least two points cannot have isolated points.  To prove the claim in the first sentence we let $X$ satisfy the hypotheses.  Let $\Sigma$ be a winning strategy for Player II.  Let $\{0, 1\}^{<\omega}$ denote the tree of functions $s: n \rightarrow \{0, 1\}$ where $n\in \omega$.  For $s \in \{0, 1\}^{<\omega}$ we let $|s|$ denote the cardinality of the domain of $s$, and for $i\in \{0, 1\}$ let $s\smallfrown i$ denote the function $t: |s|+ 1 \rightarrow \{0, 1\}$ such that $t\upharpoonright |s| = s$ and $t(|s|) = i$.

We define a scheme of open sets labeled by elements of $2^{< \omega}$ by induction.  For $|s| = 0$ we let $U_s = X$.  Assuming that we have already defined $U_s$ for all $|s| = n$ we let $U_s \supseteq V_s$ be the move dictated by $\Sigma$ under the partial run of the game $(U_{s\upharpoonright 0}, V_{s\upharpoonright 0}, U_{s\upharpoonright 1}, \ldots, U_s)$.  As $V_s \neq \emptyset$ we select disjoint nonempty open sets $U_{s \smallfrown 0}, U_{s\smallfrown 1} \subseteq V_s$.

For each element $\sigma$ of the Cantor set $\{0, 1\}^{\omega}$ we see that $(U_{\sigma\upharpoonright 0}, V_{\sigma \upharpoonright 0}, \ldots)$ is a run of the game with Player II employing strategy $\Sigma$, and so $\bigcap_{n\in \omega} U_{\sigma \upharpoonright n} \neq \emptyset$, and for distinct $\sigma \in \{0, 1\}^{\omega}$ these sets are disjoint.
\end{proof}

The relevance of the remaining material in this section will be made clear in Section \ref{modelth}.  Given a collection $\mathcal{S}$ of subsets of a set $X$, the collection $\mathcal{B} = \{X\} \cup \{\cap \mathcal{J} \mid \mathcal{J} \subseteq \mathcal{S}, |\mathcal{J}|<\infty\}$ is a basis for a topology, which we denote $\tau(\mathcal{S})$, on $X$.  The topology $\tau(\mathcal{S})$ is the smallest (under inclusion) topology on $X$ under which all elements of $\mathcal{S}$ are open.

\begin{definition}If $X$ is a topological space whose topology equals $\tau(\mathcal{S})$ then we call $\mathcal{S}$ a \emph{subbasis} for the topology on $X$.
\end{definition}

It is easy to see that if $x_0, x_1 \in X$ and for all $S \in \mathcal{S}$ we have $x_0 \in S$ implies $x_1\in S$ then for all open sets $O \in \tau(\mathcal{S})$ we have $x_0\in O$ implies $x_1 \in O$.

\begin{lemma}\label{makingHausdorff}  Suppose $X$ is a connected topological space.  Let

\begin{center}
$\hat{X} = \{(x, i, j) \in X \times \{0, 1\}^2\mid i\neq 1 \vee j\neq 1\}$
\end{center}

\noindent Let $O_1 =\{(x, i, j) \in \hat{X}\mid i = 1\}$ and $O_2 = \{(x, i, j)\in \hat{X} \mid j =1\}$ .  Endowing $\hat{X}$ with the topology given by subbasis $\{(x, i, j)\in \hat{X} \mid x\in O\}_{O \subseteq X \text{ open}} \cup \{O_1, O_2\}$, the space $\hat{X}$ is connected.
\end{lemma}

\begin{proof}  Let $\hat{X}_{00} = \hat{X} \setminus (O_1 \cup O_2)$.  It is fairly easy to see that each of $\hat{X}_{00}$, $O_1$, and $O_2$ is homeomorphic to $X$ (using projection to the $X$ coordinate).  Particularly each of $\hat{X}_{00}$, $O_1$, and $O_2$ is a connected subspace of $\hat{X}$.  Notice as well that any open set containing an element $(x, 0, 0) \in \hat{X}_{00}$ must also contain the elements $(x, 1, 0) \in O_1$ and $(x, 0, 1) \in O_2$.

If $\hat{X} = V_0 \sqcup V_1$ is a disconnection, we suppose without loss of generality that $V_0 \cap \hat{X}_{00} \neq \emptyset$.  As $\hat{X}_{00}$ is connected we know $\hat{X}_{00} \subseteq V_0$, and therefore $O_1, O_2 \subseteq V_0$.  Thus $\hat{X} = \hat{X}_{00} \cup O_1 \cup O_2 \subseteq V_0$ and we are done.
\end{proof}

\begin{lemma}\label{makingbasis}  Suppose $X$ is a connected topological space and $W_0, W_1 \subseteq X$ are open with $W_0 \cap W_1 \neq \emptyset$ (we allow $W_0 = W_1$).  Let

\begin{center}
$\hat{X} = \{(x, i) \in X \times \{0, 1\}\mid i = 1 \Rightarrow x\in W_0 \cap W_1\}$
\end{center}

\noindent  Let $O_1 = \{(x, i)\in \hat{X} \mid i = 1\}$.  Endowing $\hat{X}$ with the topology given by subbasis $\{(x, i)\in \hat{X}\mid x\in O\}_{O \subseteq X \text{ open}} \cup \{O_1\}$, the space $\hat{X}$ is connected.
\end{lemma}

\begin{proof}  Notice that the set $\hat{X}_0 = \hat{X} \setminus O_1$ is homeomorphic to $X$ (using projection to the $X$ coordinate).  Thus $\hat{X}_0$ is connected.  Any open set containing an element $(x, 0) \in (W_0 \cap W_1)\times \{0\}$ will also contain $(x, 1)$.  If $\hat{X} = V_0 \sqcup V_1$ is a disconnection, without loss of generality $V_0 \cap \hat{X}_0 \neq \emptyset$, we have $V_0 \supseteq \hat{X}_0$, and therefore $V_0 \supseteq O_1$ as well.  Therefore $\hat{X} = \hat{X}_0 \cup O_1 = V_0$.
\end{proof}

\begin{lemma}\label{splittinginhalf}  Suppose $X$ is a connected topological space and $W \subseteq X$ is a nonempty open subset.  Let

\begin{center}

$\hat{X} = \{(x, i, j)\in X \times \{0, 1\}^2 \mid [x\in W \vee i = 1 \vee j = 1]\wedge[i = j = 1 \Rightarrow x\in W]\}$

\end{center}

\noindent Let $O_1 = \{(x, i, j)\in \hat{X} \mid i = 1\}$ and $O_2 = \{(x, i, j)\in \hat{X} \mid j = 1\}$.  Endowing $\hat{X}$ with the topology given by subbasis $\{(x, i, j)\in \hat{X} \mid x\in O\}_{O \subseteq X \text{ open}} \cup \{O_1, O_2\}$, the space $\hat{X}$ is connected.
\end{lemma}

\begin{proof}  Notice that the subspaces $\hat{X}_1 = \{(x, i, j) \mid i = 1, j = 0\}$ and $\hat{X}_2 = \{(x, i, j)\mid i = 0, j = 1\}$ are each homeomorphic to $X$ (using projection to the $X$ coordinate).  Thus $\hat{X}_1$ and $\hat{X}_2$ are connected.  Also, if $x\in W$ then any open set containing $(x, 0, 0)$ must also contain $(x, i, j)$ for all choices $i, j\in \{0, 1\}$.

Suppose $\hat{X} = V_0 \sqcup V_1$ is a disconnection.  Let $x\in W$ and suppose without loss of generality that $(x, 0, 0) \in V_0$.  Then $(x, i, j)\in V_0$ for all choices $i, j\in \{0, 1\}$.  As $\hat{X}_1$ is connected we have $\hat{X}_1 \subseteq V_0$, and similarly $\hat{X}_2 \subseteq V_0$.  Then for any $x' \in W$ we have $(x', i, j) \in V_0$ for $i, j \in V_0$.  Then $\hat{X} = \hat{X}_1 \cup \hat{X}_2 \cup \{(x, i, j) \in \hat{X}\mid x\in W\} \subseteq V_0$ and we are done.
\end{proof}

\begin{lemma}\label{makingregularity}  Suppose $X$ is a connected topological space and $W \subseteq X$ is a nonempty open set.  Let

\begin{center}

$\hat{X} = \{(x, i, j) \in X \times \{0, 1\}^2 \mid [x\in W \vee j = 1] \wedge [i = 1 \Rightarrow [x\in W \wedge j = 0]]\}$

\end{center}

\noindent Let $O_1 = \{(x, i, j)\in \hat{X} \mid i = 1\}$ and $O_2 = \{(x, i, j)\in \hat{X} \mid j = 1\}$.  Endowing $\hat{X}$ with the topology given by subbasis $\{(x, i, j)\in \hat{X} \mid x\in O\}_{O \subseteq X \text{ open}} \cup \{O_1, O_2\}$, the space $\hat{X}$ is connected.

\end{lemma}

\begin{proof}  Notice that $O_2$ is homeomorphic to $X$ (using projection to the first coordinate) and therefore $O_2$ is connected.  Given $x\in X$, any open set containing $(x, 0, 0)$ will also contain $(x, 0, 1)$.  Given $x\in W$, any open set containing $(x, 0, 0)$ will also contain $(x, 1, 0)$.  Suppose $\hat{X} = V_0 \sqcup V_1$ is a disconnection.  As $W$ is nonempty we select $x\in W$.  If, without loss of generality, $(x, 0, 0) \in V_0$ then $(x, 0, 1) \in V_0$, and since $O_2$ is connected we have $O_2 \subseteq V_0$.  Then for arbitrary $x'\in W$ we have $(x', 0, 0) \in V_0$ as well.  Then $W \times \{0\}^2 \cup O_2 \subseteq V_0$, and as $W\times \{0\}^2 \subseteq V_0$ we have $W \times \{1\} \times \{0\} \subseteq V_0$.  Then $\hat{X} \subseteq V_0$.
\end{proof}

\end{section}

\begin{section}{Model Theory Preliminaries}\label{modelth}

In the course of our construction we will consider first-order logical statements regarding unary relations.  More specifically we'll construct an uncountable $I_{\Rel} \subseteq \aleph_1$ which will index a collection $\{\B_{\alpha}\}_{\alpha \in I_{\Rel}}$ of unary relations.  A collection of first-order sentences $\{\Theta_{\alpha}\}_{\alpha \in I_{\Log}}$, indexed by a subset $I_{\Log} \subseteq I_{\Rel}$, will be constructed simultaneously.  Each of the $\B_{\alpha}$ will be first introduced in some sentence $\Theta_{\alpha'}$, and each $\Theta_{\alpha'}$ will introduce one or more new relations $\B_{\alpha}$ which has not yet been mentioned.  These occur under one of the following circumstances.

\begin{enumerate}[(a)]

\item $\Theta_{\alpha_1} \equiv [\B_{\alpha_0} \cap \B_{\alpha_1} = \emptyset] \wedge [\B_{\alpha_0} \neq \emptyset]\wedge[\B_{\alpha_1} \neq \emptyset]$, where $\alpha_1 = \alpha_0 +1$ and neither $\B_{\alpha_0}$ nor $\B_{\alpha_1}$ has yet occurred in a sentence $\Theta_{\alpha'}$ with $\alpha' \in I_{\Log} \cap \alpha_1$.

\item $\Theta_{\alpha_2} \equiv [\B_{\alpha_0} \cap \B_{\alpha_1} \supseteq \B_{\alpha_2}]\wedge [\B_{\alpha_2} \neq \emptyset]$, with $\alpha_0, \alpha_1 < \alpha_2$, and $\B_{\alpha_0}$ has already occurred in some sentence $\Theta_{\alpha'}$ with $\alpha' \in I_{\Log} \cap \alpha_2$, and $\B_{\alpha_1}$ has already occurred in some sentence $\Theta_{\alpha''}$ with $\alpha'' \in I_{\Log} \cap \alpha_2$, and $\B_{\alpha_2}$ has not yet occurred in any sentence $\Theta_{\alpha'''}$ with $\alpha'''\in I_{\Log} \cap \alpha_2$.

\item $\Theta_{\alpha_1} \equiv [\B_{\alpha_0} \cap \B_{\alpha_1} \subseteq \B_{\alpha_2}] \wedge (\forall x)[x\in \B_{\alpha_0} \cup \B_{\alpha_1} \cup \B_{\alpha_2}]\wedge[\B_{\alpha_0}\neq \emptyset]\wedge[\B_{\alpha_1} \neq \emptyset]$, with $\alpha_0 + 1 = \alpha_1$, and $\alpha_2 < \alpha_0$, and $\B_{\alpha_2}$ has already occurred in some $\Theta_{\alpha'}$ with $\alpha' \in I_{\Log} \cap \alpha_0$, and neither $\B_{\alpha_0}$ nor $\B_{\alpha_1}$ has occurred in any sentence $\Theta_{\alpha''}$ with $\alpha'' \in I_{\Log} \cap \alpha_1$.

\item $\Theta_{\alpha_1} \equiv (\forall x)[x\in \B_{\alpha_2} \cup \B_{\alpha_1}]\wedge[\B_{\alpha_0} \subseteq \B_{\alpha_2}]\wedge [\B_{\alpha_0} \cap \B_{\alpha_1} = \emptyset]\wedge[\B_{\alpha_0}\neq \emptyset]\wedge[\B_{\alpha_1}\neq \emptyset]$, with $\alpha_1 = \alpha_0 + 1$ and $\alpha_2 < \alpha_0$, and $\B_{\alpha_2}$ has already occurred in a sentence $\Theta_{\alpha'}$ with $\alpha'\in I_{\Log}\cap \alpha_0$, and neither $\B_{\alpha_0}$ nor $\B_{\alpha_1}$ has occurred in a sentence $\Theta_{\alpha''}$ with $\alpha'' \in I_{\Log} \cap  \alpha_1$.
\end{enumerate}

Consider the map $h: I_{\Rel} \rightarrow I_{\Log}$ with $h(\alpha)$ being the smallest ordinal such that $\B_{\alpha}$ occurs in $\Theta_{h(\alpha)}$.  It is clear that $\alpha \leq h(\alpha)\leq \alpha + 1$.  We will let $I_{\Log, \beta} = I_{\Log} \cap \beta$, $\Gamma_{\beta} = \{\Theta_{\alpha}\}_{\alpha \in I_{\Log, \beta}}$ and $I_{\Rel, \beta} = \{\alpha \in I_{\Rel} \mid h(\alpha) <\beta\}$.  We will say that $\Gamma$ satisfies $\dagger$, or that the subset $\Gamma_{\beta}$ satisfies $\dagger$, provided that each sentence is of form (a), (b), (c), or (d) and that the order on the set of relations and the order on the set of sentences interact according to the conditions specified in each of (a), (b), (c) or (d).

\begin{remark}\label{existentialproblems}  Suppose that a structure $(\chi, \{\beth_{\alpha}\}_{\alpha \in I_{\Rel, \beta}})$ models $\B_{\alpha} \neq \emptyset$ for each $\alpha \in I_{\Rel, \beta}$ and that $\Gamma_{\beta}$ satisfies $\dagger$.  If it is not a model of  $\Gamma_{\beta}$ then there exists $x\in \chi$ which witnesses this.  More particularly suppose $\Theta_{\alpha}$ is violated.  In case $\Theta_{\alpha}$ is of type (a) then $(\exists x \in \chi)[x \in \beth_{\alpha_0}\cap \beth_{\alpha_1}]$.  In case of type (b), $(\exists x \in \chi)[x\in \beth_{\alpha_2} \setminus (\beth_{\alpha_0}\cap \beth_{\alpha_1})]$.  In case of type (c), $(\exists x \in \chi)[[x \in (\beth_{\alpha_0} \cap \beth_{\alpha_1}) \setminus \beth_{\alpha_2}]\vee [x\notin \beth_{\alpha_0}\cup\beth_{\alpha_1}\cup\beth_{\alpha_2}]]$.  In case of type (d), $(\exists x\in \chi) [[x\notin \beth_{\alpha_2}\cup\beth_{\alpha_1}]\vee[x\in \beth_{\alpha_0}\setminus \beth_{\alpha_2}]\vee[x\in \beth_{\alpha_0}\cap \beth_{\alpha_1}]]$.
\end{remark}

\begin{lemma}\label{addingpoints}  Suppose that a structure $(\chi, \{\beth_{\alpha}\}_{\alpha \in I_{\Rel, \beta}})$ models $\Gamma_{\beta}$, which itself satisfies $\dagger$.  Suppose that $(\chi^*, \{\beth_{\alpha}^*\}_{\alpha \in I_{\Rel, \beta}})$ is such that $\chi^* \supseteq \chi$ and $\beth_{\alpha}^* \supseteq \beth_{\alpha}$.  Also suppose that for each $x^* \in \chi^*$ there exists a model $(\chi^{**}, \{\beth_{\alpha}^{**}\}_{\alpha \in I_{\Rel, \beta}})$ of $\Gamma_{\beta}$ and point $x^{**}\in \chi^{**}$ such that for all $\alpha \in I_{\Rel, \beta}$ we have $x^{**} \in \beth_{\alpha}^{**}$ if and only if $x^* \in \beth_{\alpha}^*$.  Then $(\chi^*, \{\beth_{\alpha}^*\}_{\alpha \in I_{\Rel, \beta}})$ models $\Gamma_{\beta}$.
\end{lemma}

\begin{proof}  Assume the hypotheses.  Notice that $(\chi^*, \{\beth_{\alpha}^*\}_{\alpha \in I_{\Rel, \beta}})$ models $\B_{\alpha} \neq \emptyset$ for all $\alpha \in I_{\Rel, \beta}$, since $\beth_{\alpha}^* \supseteq \beth_{\alpha}$ and $(\chi, \{\beth_{\alpha}\}_{\alpha \in I_{\Rel, \beta}})$ models $\B_{\alpha}\neq \emptyset$.  By Remark \ref{existentialproblems}, if $(\chi^*, \{\beth_{\alpha}^*\}_{\alpha \in I_{\Rel, \beta}})$ fails to model $\Gamma_{\beta}$ then there are $x^* \in \chi^*$ and $\alpha' \in I_{\Log, \beta}$ such that $x$ witnesses the failure of $\Theta_{\alpha'}$.  Select a model  $(\chi^{**}, \{\beth_{\alpha}^{**}\}_{\alpha \in I_{\Rel,\beta}})$ of $\Gamma_{\beta}$ and point $x^{**}\in \chi^{**}$ such that for all $\alpha \in I_{\Rel, \beta}$ we have $x^{**} \in \beth_{\alpha}^{**}$ if and only if $x^* \in \beth_{\alpha}^*$.  Then $x^{**}$ witnesses the failure of $\Theta_{\alpha'}$ in $(\chi^{**}, \{\beth_{\alpha}^{**}\}_{\alpha \in I_{\Rel, \beta}})$, a contradiction.
\end{proof}

\begin{construction}  Suppose that $0 \leq \beta \leq \aleph_1$ and $\Gamma_{\beta}$ satisfies $\dagger$.  Consider the subset $\X_{\beta} \subseteq \{0, 1\}^{I_{\Rel, \beta}}$ given by $\sigma \in \X_{\beta}$ if and only if there exists a model $(\chi, \{\beth_{\alpha}\}_{\alpha \in I_{\Rel, \beta}})$ of $\Gamma_{\beta}$ and $x\in \chi$ such that $\sigma(\alpha) = 1 \Leftrightarrow x \in\beth_{\alpha}$.  Given $Y, Z \subseteq I_{\Rel, \beta}$ we let $\X_{\beta, Y, Z} = \{\sigma\in \X_{\beta} \mid \sigma^{-1}(0) \supseteq Y \text{ and }\sigma^{-1}(1) \supseteq Z \}$.  For each $\alpha \in I_{\Rel, \beta}$ we let $b_{\alpha}^{\beta} =  \X_{\beta, \emptyset, \{\alpha\}}$.
\end{construction}

\begin{lemma}\label{nicemodel} Let $\Gamma_{\beta}$ satisfy $\dagger$.  Then $\Gamma_{\beta}$ is consistent if and only if $(\X_{\beta}, \{b_{\alpha}^{\beta}\}_{\alpha \in I_{\Rel, \beta}})$ models $\Gamma_{\beta}$.
\end{lemma}

\begin{proof}  Certainly the reverse direction is clear.  For the forward direction we suppose that $\Gamma_{\beta}$ is consistent.  Let $(\chi, \{\beth_{\alpha}\}_{\alpha \in I_{\Rel, \beta}})$ be a model.  For each $\alpha_0 \in I_{\Rel, \beta}$ we select $x_{\alpha_0} \in \beth_{\alpha_0}$ and let $\sigma_{\alpha_0} \in X_{\beta}$ be such that $\sigma_{\alpha_0}(\alpha) = 1 \Leftrightarrow x_{\alpha_0}\in \beth_{\alpha}$.  For each $\alpha_0\in I_{\Rel, \beta}$ we see that $\sigma_{\alpha_0} \in b_{\alpha_0^{\beta}}$, and so $(\X_{\beta}, \{b_{\alpha}^{\beta}\}_{\alpha \in I_{\Rel, \beta}})$ models $\B_{\alpha_0} \neq \emptyset$ for each $\alpha_0 \in I_{\Rel, \beta}$.

Thus by Remark \ref{existentialproblems} we know that if $(\X_{\beta}, \{b_{\alpha}^{\beta}\}_{\alpha \in I_{\Rel, \beta}})$ fails to model $\Gamma_{\beta}$ then there is some point witnessing this.  For example, if $\Theta_{\alpha_1}$ is false and of type (a) then we select $\sigma \in \X_{\beta}$ for which $\sigma \in b_{\alpha_0}^{\beta} \cap b_{\alpha_1}^{\beta}$, in the notation of type (a).  Let $(\chi^*, \{\beth_{\alpha}^*\}_{\alpha \in I_{\Rel, \beta}})$ be a model with $x^* \in \chi^*$ such that $\sigma(\alpha) = 1 \Leftrightarrow  x^* \in \beth_{\alpha}^*$.  Then $x^*$ witnesses that $\Theta_{\alpha_1}$ fails in $(\chi^*, \{\beth_{\alpha}^*\}_{\alpha \in I_{\Rel, \beta}})$, a contradiction.  The comparable arguments are clear in types (b), (c), and (d).  Thus $(\X_{\beta}, \{b_{\alpha}^{\beta}\}_{\alpha \in I_{\Rel, \beta}})$ is indeed a model of $\Gamma_{\beta}$.
\end{proof}

\begin{remark}\label{emptylogic}  Let $\Gamma_{\beta}$ satisfy $\dagger$.  Notice that if $I_{\Log, \beta} = \emptyset$ then $\Gamma_{\beta} = \emptyset = I_{\Rel, \beta}$.  In particular we have $\X_{\beta} = \{0, 1\}^{\emptyset}$ consists of a single element, the empty function from $\emptyset$ to $\{0, 1\}$, and $\X_{\beta}$ is a nonempty model of $\Gamma_{\beta}$.

In case $\Gamma_{\beta}$ is inconsistent we have $I_{\Log, \beta} \neq \emptyset$ and $\X_{\beta} = \emptyset$.  If $\Gamma_{\beta}$ is consistent then either $\Gamma_{\beta} = \emptyset$, in which case $\X_{\beta}\neq \emptyset$ as we have seen, or $\Gamma_{\beta} \neq \emptyset$ in which case we are satisfying some sentence of type (a), (b), (c), or (d) and therefore $\X_{\beta} \neq \emptyset$.  Thus $\X_{\beta}$ is empty if and only if $\Gamma_{\beta}$ is inconsistent.
\end{remark}

\begin{lemma}\label{cutoff}  Let $\Gamma_{\beta}$ satisfy $\dagger$ and suppose that $\Gamma_{\beta}$ is consistent.  For each $\sigma \in \X_{\beta}$ and $\beta_0 \leq \beta$ we have $\sigma\upharpoonright I_{\Rel, \beta_0} \in \X_{\beta_0}$.
\end{lemma}

\begin{proof}  Given $\sigma \in \X_{\beta}$ we select a model $(\chi, \{\beth_{\alpha}\}_{\alpha \in I_{\Rel, \beta}})$ and $x\in \chi$ such that $x\in \beth_{\alpha} \Leftrightarrow \sigma(\alpha) = 1$.  Notice that the reduct $(\chi, \{\beth_{\alpha}\}_{\alpha \in I_{\Rel, \beta_0}})$ models $\Gamma_{\beta_0}$, and so $\sigma\upharpoonright I_{\Rel, \beta} \in \X_{\beta_0}$ follows.
\end{proof}

\begin{lemma}\label{nextstagea}  Let $\Gamma_{\beta}$ satisfy $\dagger$ and suppose that $\beta_0 \in I_{\Log, \beta}$ and that $\Theta_{\beta_0}$ is of type (a).  Then $\Gamma_{\beta_0 + 1}$ is consistent if and only if $\Gamma_{\beta_0}$ is, and in either case we have $$\X_{\beta_0 + 1} = \{(\sigma, i, j) \in \X_{\beta_0} \times \{0, 1\}^{I_{\Rel, \beta_0 + 1} \setminus I_{\Rel, \beta_0}}\mid i = 0 \vee j = 0\}$$
\end{lemma}

\begin{proof}  Let $J_a$ denote the set on the right-hand-side.  If $\Gamma_{\beta_0}$ is inconsistent then $\X_{\beta_0 + 1} = \emptyset = \X_{\beta_0}$ by Remark \ref{emptylogic} and clearly $J_a = \emptyset$ as well.  If $\Gamma_{\beta_0}$ is consistent then $\X_{\beta_0} \neq \emptyset$.  For each $\alpha \in I_{\Rel, \beta_0 + 1}$ we let $\beth_{\alpha} \subseteq J_a$ be the set $\{\rho \in J_a \mid \rho(\alpha) = 1\}$.

We claim that $(J_a, \{\beth_{\alpha}\}_{\alpha \in I_{\Rel, \beta_0 + 1}})$ is a model of $\Gamma_{\beta_0 + 1}$.  To see this, first notice that $\B_{\alpha} \neq \emptyset$ holds for all $\alpha \in I_{\Rel, \beta_0 + 1}$, for if $\alpha \in I_{\beta_0}$ we have some $\sigma \in b_{\alpha}^{\beta_0}$ and $(\sigma, 0, 0) \in \beth_{\alpha}$, and for $\alpha \in I_{\Rel, \beta_0 + 1} \setminus I_{\Rel, \beta_0}$ we take $\sigma \in \X_{\beta_0}$ (since this is nonempty) and notice that $(\sigma, 1, 0) \in \beth_{\alpha_0}$ and $(\sigma, 0, 1) \in \beth_{\alpha_1}$.  Next, it is clear that $J_a$ models $\Theta_{\beta_0}$ by how it is defined.  If $J_a$ fails to model some $\Theta_{\alpha}$ for $\alpha \in I_{\Rel, \beta_0}$ then we have some $\rho \in J_a$ as in Remark \ref{existentialproblems}, and it is easy to see that the restriction $\rho \upharpoonright I_{\Rel, \beta_0}$ witnesses that $\Theta_{\alpha}$ is false in $\X_{\beta_0}$, a contradiction.

Since $(J_a, \{\beth_{\alpha}\}_{\alpha \in I_{\Rel, \beta_0 + 1}})$ is a model of $\Gamma_{\beta_0 + 1}$ we know in particular that $\Gamma_{\beta_0 + 1}$ is consistent.  Thus by Lemma \ref{cutoff} it is clear that $\X_{\beta_0 + 1} \subseteq J_a$.  As $(J_a, \{\beth_{\alpha}\}_{\alpha \in I_{\Rel, \beta_0 + 1}})$ is a model of $\Gamma_{\beta_0 + 1}$ we have in fact $J_a \subseteq \X_{\beta_0 + 1}$ (by how $\X_{\beta_0 + 1}$ is defined) and so $J_a = \X_{\beta_0 + 1}$.
\end{proof}

\begin{lemma}\label{nextstageb}  Let $\Gamma_{\beta}$ satisfy $\dagger$ and suppose that $\beta_0 \in I_{\Log, \beta}$ and that $\Theta_{\beta_0}$ is of type (b).  Then $\Gamma_{\beta_0 + 1}$ is consistent if and only if both $\Gamma_{\beta_0}$ is consistent and $b_{\alpha_0}^{\beta_0} \cap b_{\alpha_1}^{\beta_0} \neq \emptyset$.  If $\Gamma_{\beta_0 + 1}$ is consistent then $$\X_{\beta_0 + 1} = \{(\sigma, i)\in \X_{\beta_0} \times \{0, 1\}^{I_{\Rel, \beta_0 + 1} \setminus I_{\Rel, \beta_0}}\mid i = 1 \Rightarrow [\sigma(\alpha_0) = \sigma(\alpha_1) = 1]\}$$
\end{lemma}

\begin{proof}  Suppose $\Gamma_{\beta_0 + 1}$ is consistent.  Then certainly $\Gamma_{\beta_0}$ must be consistent as well.  Suppose for contradiction that $b_{\alpha_0}^{\beta_0} \cap b_{\alpha_1}^{\beta_0} = \emptyset$.  We have $(\X_{\beta_0 + 1}, \{b_{\alpha}^{\beta_0 + 1}\}_{\alpha \in I_{\Rel, \beta_0 + 1}})$ as a model of $\Gamma_{\beta_0 + 1}$ by Lemma \ref{nicemodel}.   As $\Theta_{\beta_0}$ asserts that $\B_{\alpha_2} \neq \emptyset$, we have some $\sigma \in b_{\alpha_2}^{\beta_0 + 1} \subseteq b_{\alpha_0}^{\beta_0 + 1} \cap b_{\alpha_1}^{\beta_0 + 1}$, but now $\sigma \upharpoonright I_{\Rel, \beta_0} \in \X_{\beta_0}$ by Lemma \ref{cutoff} and this is an element of $b_{\alpha_0}^{\beta_0} \cap b_{\alpha_1}^{\beta_0}$, a contradiction.  Thus the forward directin of the first sentence is established.

Now suppose that $\Gamma_{\beta_0}$ is consistent and that $b_{\alpha_0}^{\beta_0} \cap b_{\alpha_1}^{\beta_0} \neq \emptyset$.  We let $J_b$ denote the set $\{(\sigma, i)\in \X_{\beta_0} \times \{0, 1\}^{I_{\Rel, \beta_0 + 1} \setminus I_{\Rel, \beta_0}}\mid i = 1 \Rightarrow [\sigma(\alpha_0) = \sigma(\alpha_1) = 1]\}$ and $\beth_{\alpha} = \{\rho \in J_a \mid \rho(\alpha) = 1\}$ for each $\alpha \in I_{\Rel, \beta_0 + 1}$.  We claim that $(J_b, \{\beth_{\alpha}\}_{\alpha \in I_{\Rel, \beta_0 + 1}})$ is a model of $\Gamma_{\beta_0 + 1}$.  That $\B_{\alpha} \neq \emptyset$ is satisfied for $\alpha \in I_{\Rel, \beta_0}$ is quite clear.  Since $b_{\alpha_0}^{\beta_0} \cap b_{\alpha_1}^{\beta_0} \neq \emptyset$ we select $\sigma$ in this intersection and notice that $(\sigma, 1, 1) \in \beth_{\alpha_0} \cap \beth_{\alpha_1}$, so $\B_{\beta_0} \neq \emptyset$ holds as well.  If $\Gamma_{\beta_0 + 1}$ fails in $(J_b, \{\beth_{\alpha}\}_{\alpha \in I_{\Rel, \beta_0 + 1}})$ then by Remark \ref{existentialproblems} there is a $\rho \in J_b$ witnessing this.  Arguing as in Lemma \ref{nextstagea} we obtain a contradiction.  Thus $(J_b, \{\beth_{\alpha}\}_{\alpha \in I_{\Rel, \beta_0 + 1}}) \models \Gamma_{\beta_0 + 1}$.

Now we have established that $\Gamma_{\beta_0 + 1}$ is consistent (the reverse direction of the first sentence).  We also notice by Lemma \ref{cutoff} that $\X_{\beta_0 + 1} \subseteq J_b$.  As $(J_b, \{\beth_{\alpha}\}_{\alpha \in I_{\Rel, \beta_0 + 1}})$ is a model of $\Gamma_{\beta_0 + 1}$ we have in fact $J_b \subseteq \X_{\beta_0 + 1}$, by how $\X_{\beta}$ is defined, and so $J_b = \X_{\beta_0 + 1}$.
\end{proof}

\begin{lemma}\label{nextstagec}  Let $\Gamma_{\beta}$ satisfy $\dagger$ and suppose that $\beta_0 \in I_{\Log, \beta}$ and that $\Theta_{\beta_0}$ is of type (c).  Then $\Gamma_{\beta_0 + 1}$ is consistent if and only if $\Gamma_{\beta_0}$ is, and in either case we have 

\begin{center}
$\X_{\beta_0 + 1} = \{(\sigma, i, j)\in \X_{\beta_0} \times \{0, 1\}^{I_{\Rel, \beta_0 + 1} \setminus I_{\Rel, \beta_0}}\mid [\sigma(\alpha_2) = 1 \vee i = 1 \vee j = 1] \wedge [i = j = 1 \Rightarrow \sigma(\alpha_2) = 1]\}$
\end{center}

\end{lemma}

\begin{proof}  Let $J_c$ denote the set on the right-hand-side.  If $\Gamma_{\beta_0}$ is inconsistent then $\X_{\beta_0 + 1} = \emptyset = \X_{\beta_0}$ and so $J_c = \emptyset$.  If $\Gamma_{\beta_0}$ is consistent then $\X_{\beta_0} \neq \emptyset$.  Define $\beth_{\alpha} = \{\rho \in J_c \mid \rho(\alpha) = 1\}$.  We claim that $(J_c, \{\beth_{\alpha}\}_{\alpha \in I_{\Rel, \beta_0 + 1}}) \models \Gamma_{\beta_0 + 1}$.  For each $\alpha \in I_{\Rel, \beta_0}$ we have $(J_c, \{\beth_{\alpha}\}_{\alpha \in I_{\Rel, \beta_0 + 1}}) \models \B_{\alpha} \neq \emptyset$, since we select $\sigma \in b_{\alpha}^{\beta_0}$ and notice that $(\sigma, 1, 0) \in \beth_{\alpha}$.  For $\alpha_0, \alpha_1\in I_{\beta_0 + 1} \setminus I_{\beta_0}$ we select $\sigma \in b_{\alpha_2}$ and see that $(\sigma, 1, 1) \in \beth_{\alpha_0} \cap \beth_{\alpha_1}$.  Thus $(J_c, \{\beth_{\alpha}\}_{\alpha \in I_{\Rel, \beta_0 + 1}}) \models \B_{\alpha} \neq \emptyset$ for all $\alpha \in I_{\Rel, \beta_0 + 1}$.

Certainly $(J_c, \{\beth_{\alpha}\}_{\alpha \in I_{\Rel, \beta_0 + 1}}) \models \Theta_{\beta_0}$, simply by how $J_c$ is defined, and in case an earlier $\Theta_{\alpha}$ fails we obtain a witness $\rho$ and argue as in Lemma \ref{nextstagea}.  Thus $(J_c, \{\beth_{\alpha}\}_{\alpha \in I_{\Rel, \beta_0 + 1}}) \models \Gamma_{\beta_0 + 1}$ and $J_c = \X_{\beta_0 + 1}$ follows.
\end{proof}

\begin{lemma}\label{nextstaged}  Let $\Gamma_{\beta}$ satisfy $\dagger$ and suppose that $\beta_0 \in I_{\Log, \beta}$ and that $\Theta_{\beta_0}$ is of type (d).  Then $\Gamma_{\beta_0 + 1}$ is consistent if and only if $\Gamma_{\beta_0}$ is, and in either case we have 

\begin{center}
$\X_{\beta_0 + 1} = \{(\sigma, i, j)\in \X_{\beta_0} \times \{0, 1\}^{I_{\Rel, \beta_0 + 1} \setminus I_{\Rel, \beta_0}}\mid [\sigma(\alpha_2) = 1 \vee j = 1] \wedge [i = 1 \Rightarrow [\sigma(\alpha_2) = 1 \wedge j = 0]]\}$
\end{center}

\end{lemma}

\begin{proof}  The argument, which we omit, goes more-or-less as in Lemmas \ref{nextstagea} and \ref{nextstagec}.
\end{proof}

\begin{lemma}\label{limitstage}  Let $\Gamma_{\beta_0}$ satisfy $\dagger$.  If $I_{\Log, \beta_0}$ has no largest element, then $\X_{\beta_0}$ is equal to $$\{\sigma \in  \{0, 1\}^{I_{\Rel, \beta_0}} \mid (\forall \beta_1 \in I_{\Log, \beta_0})[\sigma \upharpoonright I_{\Rel, \beta_1} \in \X_{\beta_1}]\}$$
\end{lemma}

\begin{proof}  Call the set in question $J_L$ and let $\beth_{\alpha} = \{\sigma \in J_L \mid \sigma(\alpha) = 1\}$.  If $\Gamma_{\beta_0}$ is inconsistent then $\Gamma_{\beta_1}$ is inconsistent for some $\beta_1 \in I_{\Log, \beta_0}$, so that $\X_{\beta_0} = \X_{\beta_1} = \emptyset = J_L$.  Suppose that $\Gamma_{\beta_0}$ is consistent.  Then $(\X_{\beta_0}, \{b_{\alpha}^{\beta_0}\}_{\alpha \in I_{\Rel, \beta_0}}) \models \Gamma_{\beta_0}$ by Lemma \ref{nicemodel}.  Moreover, by Lemma \ref{cutoff} we have $J_L \supseteq \X_{\beta_0}$.  This implies that $(J_L, \{\beth_{\alpha}\}_{\alpha \in I_{\Rel, \beta_0}}) \models \B_{\alpha} \neq \emptyset$ for each $\alpha \in I_{\Rel, \beta_0}$.

If $(J_L, \{\beth_{\alpha}\}_{\alpha \in I_{\Rel, \beta_0}})$ fails to model $\Gamma_{\beta_0}$ then we have a witness $\sigma$ as in  Remark \ref{existentialproblems}.  If $\sigma$ witnesses that $\Theta_{\alpha}$ is violated in $J_L$ then we select $\beta_1 \in I_{\Log, \beta_0}$ large enough that $\alpha \in I_{\Log, \beta_1}$ and notice that $\sigma\upharpoonright I_{\Rel, \beta_1}$ witnesses that $\Theta_{\alpha}$ fails in $\X_{\beta_1}$, a contradiction.  Thus $(J_L, \{\beth_{\alpha}\}_{\alpha \in I_{\Rel, \beta_0}}) \models \Gamma_{\beta_0}$, and by how $\X_{\beta_0}$ is defined we obtain $J_L \subseteq \X_{\beta_0}$ and so $J_L = \X_{\beta_0}$.
\end{proof}

\begin{lemma}\label{extension}  Let $\Gamma_{\beta}$ satisfy $\dagger$ and suppose that $\Gamma_{\beta}$ is consistent and $0 \leq \beta_0 \leq \beta$.  For each $\sigma \in \X_{\beta_0}$ there exists $\rho \in \X_{\beta}$ such that $\rho \upharpoonright I_{\Rel, \beta_0} = \sigma$.  
\end{lemma}

\begin{proof}  Let $\sigma \in \X_{\beta_0}$ be given and let $\sigma_{\beta_0} = \sigma$.  Suppose that we have defined $\sigma_{\beta_1}$ for all $\beta_0 \leq \beta_1 <\beta_3 \leq \beta$ such that if $\beta_0 \leq \beta_1 \leq \beta_2 < \beta_3$ we have $\sigma_{\beta_2} \upharpoonright I_{\Rel, \beta_1} = \sigma_{\beta_1}$.

Firstly, if $I_{\Rel, \beta_1} = I_{\Rel, \beta_3}$ for some $\beta_1 < \beta_3$ then we let $\sigma_{\beta_3} = \sigma_{\beta_1}$.  Suppose that $I_{\Rel, \beta_1} = I_{\Rel, \beta_3}$ for all $\beta_1 < \beta_3$.  If $\beta_3 = \beta_1 + 1$ then $\beta_1 \in I_{\Log, \beta_3}$ and $\X_{\beta_3}$ is computed from $\X_{\beta_1}$ by one of Lemma \ref{nextstagea}, \ref{nextstageb}, \ref{nextstagec}, or \ref{nextstaged} and in each case we see that we can produce such a $\sigma_{\beta_3}$.  If $\beta_3$ is a limit then either $I_{\Rel, \beta_3} = \emptyset$, in which case we let $\sigma_{\beta_3}$ be the unique element in $\X_{\beta_3}$, or $\beta_3 = \sup I_{\Rel, \beta_3}$ and we let

\begin{center}  $\sigma_{\beta_3}(\alpha)\begin{cases}1$ if $\sigma_{\beta_1}(\alpha) = 1$ for all $\beta_1 \in I_{\Rel, \beta_3}$ large enough that $\alpha \in I_{\Rel, \beta_1}\\0$ if $\sigma_{\beta_1}(\alpha) = 0$ for all $\beta_1 \in I_{\Rel, \beta_3}$ large enough that $\alpha \in I_{\Rel, \beta_1}\end{cases}$

\end{center}

\noindent and $\sigma_{\beta_3} \in \X_{\beta_3}$ by Lemma \ref{extension}.
\end{proof}

\begin{lemma}\label{compactsubspace}  Let $\Gamma_{\beta}$ satisfy $\dagger$.  The set $\X_{\beta}$ is compact as a subspace of $\{0, 1\}^{I_{\Rel, \beta}}$ under the product topology.
\end{lemma}

\begin{proof}  If $\Gamma_{\beta}$ is inconsistent then $\X_{\beta} = \emptyset$, which is compact.  Otherwise, supposing that $\sigma_0 \in \{0, 1\}^{I_{\Rel, \beta}} \setminus \X_{\beta}$ we let $J = \X_{\beta} \cup\{\sigma_0\}$ and define $\beth_{\alpha} = \{\sigma \in J \mid \sigma(\alpha) = 1\}$ for each $\alpha \in I_{\Rel, \beta}$.  As $J \supseteq \X_{\beta}$ we see that $(J, \{\beth_{\alpha}\}_{\alpha \in I_{\Rel, \beta}}) \models \B_{\alpha} \neq \emptyset$ for each $\alpha \in I_{\Rel, \beta}$.  Since $J \neq \X_{\beta}$ we know that $(J, \{\beth_{\alpha}\}_{\alpha \in I_{\Rel, \beta}})$ is not a model of $\Gamma_{\beta}$ (by how $\X_{\beta}$ is defined).  By Remark \ref{existentialproblems} we have a witness for this, and it is easy to see that this witness is $\sigma_0$.  Thus if $\Theta_{\alpha}$, say of type (b), is violated we have that $\{\sigma \in \{0, 1\}^{I_{\Rel, \beta}} \mid \sigma\upharpoonright \{\alpha_0, \alpha_1, \alpha_2\} = \sigma_0\upharpoonright\{\alpha_0, \alpha_1, \alpha_2\}\} \subseteq \{0, 1\}^{I_{\Rel, \beta}} \setminus \X_{\beta}$.  The other cases in types (a), (c), (d) are comparable.  Thus $\X_{\beta}$ is closed in $\{0, 1\}^{I_{\Rel, \beta}}$, and as the latter space is compact we are finished.
\end{proof}

\begin{lemma}\label{compactconnected}  Let $\Gamma_{\beta}$ satisfy $\dagger$.  The topological space $(\X_{\beta}, \tau(\{b_{\alpha}^{\beta}\}_{\alpha \in I_{\Rel, \beta}}))$ is compact and connected.
\end{lemma}

\begin{proof}  If $\Gamma_{\beta}$ is inconsistent then $\X_{\beta}$ is empty, and therefore compact and connected.  Suppose that $\Gamma_{\beta}$ is consistent.  The space $(\X_{\beta}, \tau(\{b_{\alpha}^{\beta}\}_{\alpha \in I_{\Rel, \beta}}))$ is compact because $\X_{\beta}$ is compact as a subspace of $\{0, 1\}^{I_{\Rel, \beta}}$ under the product topology, and each $b_{\alpha}^{\beta}$ is the intersection of an open subset in $\{0, 1\}^{I_{\Rel, \beta}}$ with $\X_{\beta}$.

We prove connectedness by induction on $0 \leq \beta_0 \leq \beta$.  For $\beta_0 = 0$, or more generally if $I_{\Log, \beta_0} = \emptyset$, then $\X_{\beta_0}$ is a single point, and so the space is necessarily connected.  Suppose that $\X_{\beta_1}$ is connected for all $\beta_1 < \beta_0$.  Suppose further that $I_{\Log, \beta_0}$ has a maximal element, $\beta_1$.  If $\beta_0 = \beta_1 + 1$, then $\Theta_{\beta_1}$ is of type (a), (b), (c), or (d), and so $\X_{\beta_0}$ is defined from $\X_{\beta_1}$ by Lemma \ref{nextstagea}, \ref{nextstageb}, \ref{nextstagec}, or \ref{nextstaged} respectively.  Then we conclude that $\X_{\beta_0}$ is connected by Lemma \ref{makingHausdorff}, \ref{makingbasis}, \ref{splittinginhalf}, or \ref{makingregularity} respectively.  Otherwise we have $\X_{\beta_0} = \X_{\beta_1 + 1}$.

Suppose that $I_{\Log, \beta_0}$ is nonempty and does not have a maximal element.  Suppose for contradiction that $\X_{\beta_0} = V_0 \sqcup V_1$ is a nontrivial disconnection.  Let $V_0 = \bigcup_{j\in J_0}(b_{\alpha_{0, j}}^{\beta_0} \cap \cdots \cap b_{\alpha_{n_j, j}}^{\beta_0})$ and similarly $V_1 = \bigcup_{j\in J_1}(b_{\alpha_{0, j}}^{\beta_0} \cap \cdots \cap b_{\alpha_{n_j, j}}^{\beta_0})$.  Since $\X_{\beta_0}$ is compact we may assume that both $J_0$ and $J_1$ are finite.  Select $\beta_1 < \beta_0$ which is large enough that $\alpha_{i, j} \in I_{\Rel, \beta_1}$ for all $j \in J_0 \cup J_1$ and $0 \leq i \leq n_j$.  Let $V_0^{\beta_1} = \bigcup_{j\in J_0}(b_{\alpha_{0, j}}^{\beta_1} \cap \cdots \cap b_{\alpha_{n_j, j}}^{\beta_1})$ and $V_1^{\beta_1} = \bigcup_{j\in J_1}(b_{\alpha_{0, j}}^{\beta_1} \cap \cdots \cap b_{\alpha_{n_j, j}}^{\beta_1})$.

Notice that neither $V_0^{\beta_1}$ nor $V_1^{\beta_1}$ is empty, for if, say, $\sigma \in b_{\alpha_{0, j}}^{\beta_0} \cap \cdots \cap b_{\alpha_{n_j, j}}^{\beta_0}$ for some $j \in J_0$, then $\sigma \upharpoonright I_{\Rel, \beta_1} \in b_{\alpha_{0, j}}^{\beta_1} \cap \cdots \cap b_{\alpha_{n_j, j}}^{\beta_1} \subseteq V_0^{\beta_1}$, and a similar argument applies for $V_1^{\beta_1}$.  

The sets $V_0^{\beta_1}$ and $V_1^{\beta_1}$ are obviously open in $\X_{\beta_0}$.  If $\sigma \in \X_{\beta_0} \setminus (V_0^{\beta_1} \cup V_1^{\beta_1})$ then by Lemma \ref{extension} select $\rho \in \X_{\beta_0}$ such that $\rho \upharpoonright I_{\Rel, \beta_1} = \sigma$, but clearly $\rho \in \X_{\beta_0}\setminus (V_0 \cup V_1)$, which cannot be.  If $\sigma \in V_0^{\beta_1} \cap V_1^{\beta_1}$ then again we select $\rho \in \X_{\beta_0}$ with $\rho \upharpoonright I_{\Rel, \beta_1} = \sigma$ and it is easy to see that $\rho \in V_0 \cap V_1$, which is also impossible.  Thus $\X_{\beta_1} = V_0^{\beta_1} \sqcup V_1^{\beta_1}$ is a nontrivial disconnection, contradicting the induction assumption.
\end{proof}

\begin{definitions}  Let $\Gamma_{\beta}$ satisfy $\dagger$.  Given a model $(\chi, \{\beth_{\alpha}\}_{\alpha \in I_{\Rel, \beta}})$ of $\Gamma_{\beta}$ there is a function $\Coor_{\beta}: \chi \rightarrow \X_{\beta}$ (which we'll call the \textit{coordinates function}) given by $\Coor_{\beta}(x) = \sigma$, where $\sigma(\alpha) = 1$ if and only if $x\in \beth_{\alpha}$.  We'll say that a model $(\chi, \{\beth_{\alpha}\}_{\alpha \in I_{\Rel, \beta}})$ of $\Gamma_{\beta}$ is \textit{Boolean saturated} if for every finite $Y, Z \subseteq I_{\Rel, \beta}$ such that $\X_{\beta, Y, Z} \neq \emptyset$ the set $\Coor_{\beta}^{-1}(\X_{\beta, Y, Z})$ is infinite.
\end{definitions}

\begin{lemma} \label{coord}  Let $\Gamma_{\beta}$ satisfy $\dagger$.  The coordinates function $\Coor_{\beta}: \chi \rightarrow \X_{\beta}$ is continuous, where $\chi$ is given the topology $\tau(\{\beth_{\alpha}\}_{\alpha \in I_{\Rel, \beta}})$ and $\X_{\beta}$ is given the topology $\tau(\{b_{\alpha}^{\beta}\}_{\alpha \in I_{\Rel, \beta}})$.  The function $\Coor_{\beta}$ is an embedding if injective.  The image $\Coor_{\beta}(\chi)$ is dense in $\X_{\beta}$ if $\chi$ is Boolean saturated.
\end{lemma}

\begin{proof} For the first two sentences we have $x\in \beth_{\alpha_0} \cap \cdots \cap \beth_{\alpha_n}$ if and only if $\Coor_{\beta}(x) \in b_{\alpha_0}^{\beta}\cap \cdots \cap b_{\alpha_n}^{\beta}$, and this is sufficient.  For the third sentence, if $\chi$ is Boolean saturated then $\Coor_{\beta}(\chi)$ is dense in the set $\X_{\beta}$ under the subspace topology inherited by $\{0, 1\}^{I_{\Rel, \beta}}$, and since this topology is finer than $\tau(\{b_{\alpha}^{\beta}\}_{\alpha \in I_{\Rel, \beta}})$ we are done.
\end{proof}

\begin{lemma}\label{refinecasea}  Let $\beta_1 < \aleph_1$ and $\Gamma_{\beta_1}$ satisfy $\dagger$.  Suppose $\Gamma_{\beta_1}$ is consistent and that $\beta_0$ is the largest element in $I_{\Log, \beta_1}$.  Let $(\chi, \{\beth_{\alpha}\}_{\alpha \in I_{\Rel, \beta_0}})$ be a Boolean saturated model of $\Gamma_{\beta_0}$.

If $\Theta_{\beta_0}$ is of type (a) and $x_0, x_1 \in \chi$ are distinct, we can define $\beth_{\alpha_0}, \beth_{\alpha_1} \subseteq \chi$ so that $(\chi, \{\beth_{\alpha}\}_{\alpha \in I_{\Rel, \beta_1}})$ is a Boolean saturated model of $\Gamma_{\beta_1}$ and $x_0 \in \beth_{\alpha_0}$ and $x_1 \in \beth_{\alpha_1}$.
\end{lemma}

\begin{proof}  Let $\{(Y_n, Z_n)\}_{n\in \omega}$ be an enumeration of all ordered pairs $(Y, Z)$ with $Y, Z \subseteq I_{\Rel, \beta_0}$ finite and $\X_{\gamma_0, Y, Z} \neq \emptyset$, and such that for each $n\in \omega$ there are infinitely many $n'$ in each of  $3\omega$, $3\omega +1$, and $3\omega + 2$ such that $(Y_n, Z_n) = (Y_{n'}, Z_{n'})$.  We'll inductively construct a sequence of finite sets $F_{-1} \subseteq F_0 \subseteq F_1 \subseteq \cdots$ as well as assign points to $\beth_{\alpha_0}, \beth_{\alpha_1}, \chi \setminus (\beth_{\alpha_0} \cup \beth_{\alpha_1})$.

Let $x_0 \in \beth_{\alpha_0}$ and $x_1 \in \beth_{\alpha_1}$ and let $F_{-1} = \{x_0, x_1\}$.  If $n \geq 0$ and we have already defined $F_{n-1}$ we select $x\in \Coor_{\beta_0}^{-1}(\X_{Y_n, Z_n, \beta_0}) \setminus F_{n-1}$ and assign $x$ in the following way:

\begin{center}
$\begin{cases} x\in \beth_{\alpha_0}$ if $n\in 3\omega\\ x\in \beth_{\alpha_1}$ if $n\in 3\omega+1\\ x\in \chi\setminus (\beth_{\alpha_0}\cup\beth_{\alpha_1})$ if $n \in 3\omega +2 \end{cases}$
\end{center}

\noindent and let $F_n = F_{n-1}\cup \{x\}$.  For $x\in \chi \setminus \bigcup_{n\in \omega} F_n$ we let $x\in \beth_{\alpha_0}$.  The check that  $(\chi, \{\beth_{\alpha}\}_{\alpha \in I_{\Rel, \beta_1}})$ is a Boolean saturated model of $\Gamma_{\beta_1}$ is straightforward.
\end{proof}

\begin{lemma}\label{refinecaseb}  Let $\beta_1 < \aleph_1$ and $\Gamma_{\beta_1}$ satisfy $\dagger$.  Suppose $\Gamma_{\beta_1}$ is consistent and that $\beta_0$ is the largest element in $I_{\Log, \beta_1}$.  Let $(\chi, \{\beth_{\alpha}\}_{\alpha\in I_{\Rel, \beta_0}})$ be a Boolean saturated model of $\Gamma_{\beta_0}$.

If $\Theta_{\beta_0}$ is of type (b) and $x_2 \in \beth_{\alpha_0} \cap \beth_{\alpha_1}$, we can define $\beth_{\alpha_2} \subseteq \chi$ so that $(\chi, \{\beth_{\alpha}\}_{\alpha \in I_{\Rel, \beta_1}})$ is a Boolean saturated model of $\Gamma_{\beta_1}$ and $x_2\in \beth_{\alpha_2}$.
\end{lemma}

\begin{proof}   Let $\{(Y_n, Z_n)\}_{n\in \omega}$ be an enumeration of all ordered pairs $(Y, Z)$ with $Y, Z \subseteq I_{\Rel, \beta_0}$ finite and $\X_{\gamma_0, Y, Z} \neq \emptyset$, and such that for each $n\in \omega$ there are infinitely many $n'$ in each of  $2\omega$ and $2\omega +1$ such that $(Y_n, Z_n) = (Y_{n'}, Z_{n'})$.  We'll inductively construct a sequence of finite sets $F_{-1} \subseteq F_0 \subseteq F_1 \subseteq \cdots$ as well as assign points to $\beth_{\alpha_2}$ and to $\chi \setminus \beth_{\alpha_2}$.

Let $x_2\in \beth_{\alpha_2}$ and let $F_{-1} = \{x_2\}$.  If $n\geq 0$ and we have already defined $F_{n-1}$ we select $x\in \Coor_{\beta_0}^{-1}(\X_{\beta_0, Y_n, Z_n}) \setminus F_{n-1}$ and assign $x$ in the following way:

\begin{center}
$\begin{cases}x\in \beth_{\alpha_2}$ if $x\in \beth_{\alpha_0}\cap \beth_{\alpha_1}$ and $n\in 2\omega\\x\in \chi\setminus \beth_{\alpha_2}$ if $x\in \beth_{\alpha_0}\cap \beth_{\alpha_1}$ and $n\in 2\omega + 1\\x\in \chi\setminus \beth_{\alpha_2}$ if $x\notin \beth_{\alpha_0}\cap \beth_{\alpha_1}  \end{cases}$
\end{center}

\noindent and let $F_n = F_{n-1} \cup \{x\}$.  Assign $x\in \chi \setminus \bigcup_{n\in \omega} F_n$ by the rule

\begin{center}
$\begin{cases}x\in \beth_{\alpha_2}$ if $x\in \beth_{\alpha_0}\cap \beth_{\alpha_1}\\x\in \chi \setminus \beth_{\alpha_2}$ otherwise  $ \end{cases}$
\end{center}

\noindent Again the check that the conclusion is fulfilled is left to the reader.
\end{proof}

\begin{lemma}\label{refinecasec}  Let $\beta_1 < \aleph_1$ and $\Gamma_{\beta_1}$ satisfy $\dagger$.  Suppose $\Gamma_{\beta_1}$ is consistent and that $\beta_0$ is the largest element in $I_{\Log, \beta_1}$.  Let $(\chi, \{\beth_{\alpha}\}_{\alpha \in I_{\Rel, \beta_0}})$ be a Boolean saturated model of $\Gamma_{\beta_0}$.

If $\Theta_{\beta_0}$ is of type (c) and $x_0, x_1$ are distinct points in $\chi \setminus \beth_{\alpha_2}$, we can define $\beth_{\alpha_0}, \beth_{\alpha_1} \subseteq \chi$ so that $(\chi, \{\beth_{\alpha}\}_{\alpha \in I_ {\Rel, \beta_1}})$ is a Boolean saturated model of $\Gamma_{\beta_1}$ and $x_0 \in \beth_{\alpha_0}$ and $x_1 \in \beth_{\alpha_1}$.
\end{lemma}

\begin{proof}  Let $\{(Y_n, Z_n)\}_{n\in \omega}$ be an enumeration of all ordered pairs $(Y, Z)$ with $Y, Z \subseteq I_{\Rel,\beta_0}$ finite and $\X_{\beta_0, Y, Z} \neq \emptyset$, and such that for each $n\in \omega$ there are infinitely many $n'$ in each of  $4\omega$, $4\omega +1$, $4\omega +2$, and  $4\omega + 3$ such that $(Y_n, Z_n) = (Y_{n'}, Z_{n'})$.  We'll inductively construct a sequence of finite sets $F_{-1} \subseteq F_0 \subseteq F_1 \subseteq \cdots$ as well as assign points to $\beth_{\alpha_0}\setminus \beth_{\alpha_1}$, $\beth_{\alpha_1}\setminus \beth_{\alpha_0}$, $\beth_{\alpha_0}\cap \beth_{\alpha_1}$ and $\chi \setminus (\beth_{\alpha_0} \cup \beth_{\alpha_1})$.

Let $x_0 \in \beth_{\alpha_0} \setminus \beth_{\alpha_1}$, $x_1 \in \beth_{\alpha_1}\setminus \beth_{\alpha_0}$ and set $F_{-1} = \{x_0, x_1\}$.  If $n \geq 0$ we select $x\in \Coor_{\beta_0}^{-1}(\X_{\beta_0, Y_n, Z_n}) \setminus F_{n-1}$ and assign

\begin{center}
$\begin{cases}x \in  \beth_{\alpha_0} \setminus \beth_{\alpha_1}$ if $x\notin \beth_{\alpha_2}$ and $n\in 2\omega\\ x\in \beth_{\alpha_1}\setminus \beth_{\alpha_0}$ if $x\notin \beth_{\alpha_2}$ and $n\in 2\omega +1\\ x\in \beth_{\alpha_0}\cap \beth_{\alpha_1}$ if $x\in \beth_{\alpha_2}$ and $n\in 4\omega\\ x\in \beth_{\alpha_0} \setminus \beth_{\alpha_1}$ if $x\in \beth_{\alpha_2}$ and $n\in 4\omega +1\\ x\in \beth_{\alpha_1}\setminus \beth_{\alpha_0}$ if $x\in \beth_{\alpha_2}$ and $n\in 4\omega +2\\ x\in \chi \setminus (\beth_{\alpha_0}\cup\beth_{\alpha_1})$ if $x\in \beth_{\alpha_2}$ and $n\in 4\omega +3  \end{cases}$
\end{center}

\noindent and let $F_n = F_{n-1}\cup \{x\}$.  Assign $x\in \chi \setminus \bigcup_{n\in \omega}F_n$ by $x\in \beth_{\alpha_0}\setminus \beth_{\alpha_1}$.
\end{proof}

\begin{lemma}\label{refinecased}   Let $\beta_1 < \aleph_1$ and $\Gamma_{\beta_1}$ satisfy $\dagger$.  Suppose $\Gamma_{\beta_1}$ is consistent and that $\beta_0$ is the largest element in $I_{\Log, \beta_1}$.  Let $(\chi, \{\beth_{\alpha}\}_{\alpha \in I_{\Rel, \beta_0}})$ be a Boolean saturated model of $\Gamma_{\beta_0}$.

If $\Theta_{\beta_0}$ is of type (d) and $x_2 \in \beth_{\alpha_2}$, we can define $\beth_{\alpha_0}, \beth_{\alpha_1} \subseteq \chi$ so that $(\chi, \{\beth_{\alpha}\}_{\alpha \in I_{\Rel, \beta_1}})$ is a Boolean saturated model of $\Gamma_{\beta_1}$ and $x_2\in \beth_{\alpha_0}$.
\end{lemma}

\begin{proof}   Let $\{(Y_n, Z_n)\}_{n\in \omega}$ be an enumeration of all ordered pairs $(Y, Z)$ with $Y, Z \subseteq I_{\Rel, \beta_0}$ finite and $\X_{\beta_0, Y, Z} \neq \emptyset$, and such that for each $n\in \omega$ there are infinitely many $n'$ in each of  $3\omega$, $3\omega +1$, and $3\omega +2$ such that $(Y_n, Z_n) = (Y_{n'}, Z_{n'})$.  We'll inductively construct a sequence of finite sets $F_{-1} \subseteq F_0 \subseteq F_1 \subseteq \cdots$ as well as assign points to $\beth_{\alpha_1}\setminus \beth_{\alpha2}$, $\beth_{\alpha_2}\setminus \beth_{\alpha_1}$, $\beth_{\alpha_1}\cap \beth_{\alpha_2}$ and $\chi \setminus (\beth_{\alpha_1} \cup \beth_{\alpha_2})$.

Let $x_2\in \beth_{\alpha_0}$ and $F_{-1} = \{x_2\}$.  For $n \geq 0$ we select $x\in \Coor_{\beta_0}^{-1}(\X_{\beta_0, Y_n, Z_n}) \setminus F_{n-1}$ and assign

\begin{center}
$\begin{cases}x \in  \chi \setminus(\beth_{\alpha_0} \cup \beth_{\alpha_1})$ if $x\in \beth_{\alpha_2}$ and $n\in 3\omega\\ x\in \beth_{\alpha_0} \setminus \beth_{\alpha_1}$ if $x\in \beth_{\alpha_2}$ and $n\in 3\omega+1\\ x\in \beth_{\alpha_1}\setminus \beth_{\alpha_0}$ if $x\in \beth_{\alpha_2}$ and $n\in 3\omega +2\\ x\in \beth_{\alpha_1}\setminus \beth_{\alpha_0}$ if $x\notin \beth_{\alpha_2} \end{cases}$
\end{center}

\noindent and let $F_n = F_{n-1} \cup\{x\}$.  For $x\in \chi \setminus \bigcup_{n\in \omega} F_n$ we assign $x\in \beth_{\alpha_1} \setminus \beth_{\alpha_0}$.
\end{proof}

\begin{definitions}(see \cite{Jec})  A subset $E$ of an ordinal $\alpha$ is \emph{bounded in $\alpha$} provided there exists $\beta< \alpha$ which is an upper bound on $E$.  A subset $C$ of $\aleph_1$ is \emph{club} if it is unbounded in $\aleph_1$ and closed under the order topology in $\aleph_1$.  The intersection of two club sets in $\aleph_1$ is again a club set.  A subset $E$ of $\aleph_1$ is \emph{stationary} if it has nonempty intersection with every club subset of $\aleph_1$.  The intersection of a club set and a stationary set is again stationary.
\end{definitions}

We quote Jensen's $\diamondsuit$ principle (see \cite[Lemma 6.5]{Jen} or \cite[Theorem 13.21]{Jec}):

\begin{center}

\noindent $\diamondsuit$:  There exists a sequence $\{S_{\alpha}\}_{\alpha< \aleph_1}$ such that $S_{\alpha} \subseteq \alpha$ and for any $J\subseteq \aleph_1$ the set $\{\alpha < \aleph_1 \mid J\cap \alpha = S_{\alpha}\}$ is stationary in $\aleph_1$.

\end{center}

The sequence which is claimed in $\diamondsuit$ is often called a $\diamondsuit$\emph{-sequence}.  Jensen proved that $\diamondsuit$ holds in G\"odel's constructible universe $L$.  As a result,  if ZFC is consistent then so is ZFC + $\diamondsuit$.

\begin{remark}\label{CH}  An obvious consequence of $\diamondsuit$ is that $\aleph_1 = 2^{\aleph_0}$, since the function $2^{\omega} \rightarrow \aleph_1$ given by $J \mapsto \min \{\omega \leq \alpha <\aleph_1 \mid S_{\alpha} = J\}$ is injective.  Another well known consequence is the following:
\end{remark}

\begin{lemma}\label{twosetsnotone} (ZFC + $\diamondsuit$)  Let $T\subseteq \aleph_1$ be a club. There exists a sequence $\{(S_{\alpha, 0}, S_{\alpha, 1})\}_{\alpha \in T}$ such that for any pair $(J_0, J_1)$ of sets such that $J_0, J_1 \subseteq \aleph_1$ the set $\{\alpha \in T\mid \alpha\cap J_0 = S_{\alpha, 0} \text{ and } \alpha \cap J_1 = S_{\alpha, 1}\}$ is stationary.
\end{lemma}

\begin{proof}  Give the product $\aleph_1 \times \{0, 1\}$ the lexicographic order (initially comparing left coordinates).  The unique order isomorphism $f: \aleph_1 \times \{0, 1\} \rightarrow \aleph_1$ has $f(\alpha, 0) = \alpha$ for each limit ordinal $\alpha <\aleph_1$. For $i\in \{0, 1\}$ define $f_i: \aleph_1 \rightarrow \aleph_1$ by $f_i(\alpha) = f(\alpha, i)$, thus $f_0(\aleph_1) \sqcup f_1(\aleph_1) = \aleph_1$.  Moreover for each limit ordinal $\alpha$ we have $f_i(J \cap \alpha) = f_i(J) \cap \alpha$.

Let $\{S_{\alpha}\}_{\alpha < \aleph_1}$ be a $\diamondsuit$-sequence on $\aleph_1$.  Let $T'\subseteq T$ be the set of limit ordinals in $T$; this is also a club.  Let sequence $\{(S_{\alpha, 0}, S_{\alpha, 1})\}_{\alpha \in T'}$ be defined by $S_{\alpha, i} = f_i^{-1}(S_{\alpha})$.  Given $J_0, J_1 \subseteq \aleph_1$ we let $J = f_0(J_0) \sqcup f_1(J_1)$.  We have

\begin{center}

$\{\alpha \in T' \mid J_0\cap \alpha = S_{\alpha, 0} \text{ and }J_1 \cap \alpha = S_{\alpha, 1}\}$

$= \{\alpha \in T' \mid f_0(J_0 \cap \alpha) \sqcup f_1(J_1 \cap \alpha) = S_{\alpha}\}$

$= \{\alpha \in T' \mid (f_0(J_0)\cap \alpha)\sqcup (f_1(J_1)\cap \alpha) = S_{\alpha}\}$

$= \{\alpha \in T' \mid J\cap \alpha = S_{\alpha}\}$

$= T' \cap \{\alpha < \aleph_1 \mid J \cap \alpha = S_{\alpha}\}$

\end{center}

\noindent and this set is stationary as the intersection of a club and a stationary.  Letting $(S_{\alpha, 0}, S_{\alpha, 1}) = (\emptyset, \emptyset)$ for $\alpha \in T\setminus T'$ we obtain the desired sequence.

\end{proof}

\end{section}

\begin{section}{The Construction}\label{construction}

We will be defining collections and sets by induction.  These will include sets $I_{\Log}, I_{\Rel}, I_{\Conn} \subseteq \aleph_1$, a collection $\Sequ$ of ordered pairs of strictly increasing $\omega$-sequences, countable collections $\{\Bo^{\beta}\}_{\beta< \aleph_1}$ of subsets of $\aleph_1$, and a collection of first order sentences for unary predicates $\Gamma$.  We'll explain the interactions among these sets.

The set $\Gamma$ will be indexed by $I_{\Log}$: $\Gamma = \{\Theta_{\alpha}\}_{\alpha \in I_{\Log}}$ and $\Gamma$ will satisfy $\dagger$.  The set $I_{\Rel}$ will satisfy $I_{\Rel} \supseteq I_{\Log}$ and will index the set $\{\B_{\alpha}\}_{\alpha \in I_{\Rel}}$ of unary predicate symbols which appear in $\Gamma$.  Each of the statements $\Theta_{\alpha}$ which appears in $\Gamma$ will satisfy one of (a) - (d) of $\dagger$, and the interactions with the symbols $\{\B_{\alpha}\}_{\alpha \in I_{\Rel}}$ will be as determined in $\dagger$.  Thus, as in Section \ref{modelth} we let $h: I_{\Rel} \rightarrow I_{\Log}$ map $\alpha$ to the minimal $\alpha'$ such that $\B_{\alpha}$ appears in $\Theta_{\alpha'}$.  As noted, this will satisfy $\alpha \leq h(\alpha) \leq \alpha +1$.  Again, for each $0 \leq \beta < \aleph_1$ we let $I_{\Log, \beta} = I_{\Log} \cap \beta$ and $I_{\Rel, \beta} = \{\alpha \in I_{\Rel} \mid h(\alpha)<\beta\}$, noting that this will sometimes be a proper subset of $I_{\Rel} \cap \beta$.

For each $0 \leq \beta < \aleph_1$ the collection $\Bo^{\beta}$ will be indexed by $I_{\Rel, \beta}$: $\Bo^{\beta} = \{B_{\alpha}^{\beta}\}_{\alpha \in I_{\Rel, \beta}}$.  Importantly, the superscript $\beta$ in this case is only an index; it is not a set of functions from $\beta$ to a set $\Bo$ or $B_{\alpha}$.  Moreover $B_{\alpha}^{\beta} \subseteq \beta$ will hold for each $\alpha\in I_{\Rel, \beta}$.  Also we will have $B_{\alpha}^{\beta_0}  = B_{\alpha}^{\beta} \cap \beta_0$ whenever $\alpha \in I_{\Rel, \beta_0}$ and $\beta_0 \leq \beta < \aleph_1$.

We will let $I_{\Conn, \beta} = I_{\Conn} \cap \beta$.  The elements of $\Sequ$ will be indexed by $I_{\Conn} \times \{0, 1\}$: $\Sequ = \{(\{s_{\alpha, 0, q}\}_{q \in \omega}, \{s_{\alpha, 1, q}\}_{q \in \omega})\}_{\alpha \in I_{\Conn}}$.  For $0 \leq \beta<\aleph_1$ and $\alpha \in I_{\Conn, \beta}$, both of $\{s_{\alpha, 0, q}\}_{q\in \omega}$ and $\{s_{\alpha, 1, q}\}_{q \in \omega}$ will be strictly increasing such that $\alpha = \sup_{q \in \omega} s_{\alpha, 0, q} = \sup_{q\in \omega} s_{\alpha, 1, q}$.

Now we enumerate the inductive hypotheses explicitly.  Let $T \subseteq \aleph_1$ be the second derived subset (i.e the set of limit points of limit points in $\aleph_1$ under the order topology).  As the set $T$ is club in $\aleph_1$ we let $\{(S_{\alpha, 0}, S_{\alpha, 1})\}_{\alpha \in T}$ be a sequence as in Lemma \ref{twosetsnotone}.  Let $T^* = T \cup \{0\}$.  Each ordinal $\gamma < \aleph_1$ may be uniquely written as $\gamma = t + \omega l + n$ where $t \in T^*$ and $l, n\in \omega$.  Let $R$ denote the set of ordinals $\gamma <\aleph_1$ which are of form $\gamma = t + \omega 6k +n$ with $\gamma \notin \omega$ and $n \geq 2$.  As we are assuming $\diamondsuit$ we have $\aleph_1 = 2^{\aleph_0}$.  Thus, let $f: R \rightarrow \bigcup_{\delta < \aleph_1} \{0, 1\}^{\delta}$ be a function such that each element of the codomain has uncountable preimage.

The following will hold for all $0 \leq \beta < \aleph_1$.

\begin{enumerate}[(i)]

\item The set $I_{\Log, \beta}$ is such that $I_{\Log, \beta_0} = I_{\Log, \beta} \cap \beta_0$ for each $\beta_0 \leq \beta$.  Also, $\Gamma_{\beta} = \{\Theta_{\alpha}\}_{\alpha \in I_{\Log, \beta}}$ and satisfies $\dagger$.  Also, $\Gamma_{\beta_0} = \{\Theta_{\alpha}\}_{\alpha <\beta_0} \subseteq \Gamma_{\beta}$ for $\beta_0 \leq \beta$.

\item The set $I_{\Rel, \beta}$ is such that $I_{\Rel, \beta_0} = I_{\Rel, \beta} \cap \beta_0$ for each $\beta_0 \in I_{\Log, \beta}$.  Also, the collection $\Bo^{\beta} = \{B_{\alpha}^{\beta}\}_{\alpha \in I_{\Rel, \beta}}$ is such that $B_{\alpha}^{\beta} \subseteq \beta$, and for $\beta_0 \leq \beta$ and $\alpha \in I_{\Rel, \beta_0}$ we have $B_{\alpha}^{\beta_0} = B_{\alpha}^{\beta} \cap \beta_0$.

\item $\Bo^{\beta}\models \Gamma_{\beta}$.

\item If $\omega \leq \beta$ then $(\omega, \{\omega \cap B_{\alpha}^{\beta}\}_{\alpha \in I_{\Rel, \beta}})$ is a Boolean saturated model of $\Gamma_{\beta}$.

\item The set $I_{\Conn, \beta}$ is such that $I_{\Conn, \beta_0} = I_{\Conn, \beta} \cap \beta_0$ for each $\beta_0 \leq \beta$, and $I_{\Conn, \beta} \subseteq T$.  The set $\Sequ_{\beta}$ consists of ordered pairs of strictly increasing sequences $$\Sequ_{\beta} = \{(\{s_{\alpha, 0, q}\}_{q \in \omega}, \{s_{\alpha, 1, q}\}_{q \in \omega})\}_{\alpha \in I_{\Conn, \beta}}$$ such that $\sup_{q \in \omega} s_{\alpha, 0, q} = \alpha = \sup_{q \in \omega} s_{\alpha, 1, q}$ and $\{s_{\alpha, 0, q}\}_{q\in \omega}, \{s_{\alpha, 1, q}\}_{q \in \omega} \subseteq \alpha \setminus (T \cup \omega)$.  Also, $\Sequ_{\beta_0} = \{(\{s_{\alpha, 0, q}\}_{q\in \omega}, \{s_{\alpha, 1, q}\}_{q \in \omega})\}_{\alpha \in I_{\Conn, \beta_0}} \subseteq \Sequ_{\beta}$ for each $\beta_0 \leq \beta$.

\item For each $\alpha_0 \in I_{\Conn, \beta}$ and $\alpha_1 \in I_{\Rel, \beta}$ we have that $\alpha_0 \in B_{\alpha_1}^{\beta}$ implies that there exists $N \in \omega$ for which $\{s_{\alpha_0, 0, q}\}_{q \geq N}, \{s_{\alpha_0, 1, q}\}_{q \geq N} \subseteq B_{\alpha_1}^{\beta}$.

\end{enumerate}

We'll consider cases in the construction: N (for ``n''atural number), L (for ``l''imit ordinal), A, B, C, D (corresponding to introduction of logical sentences and basis elements according to (a), (b), (c), (d) in $\dagger$), and also Cases E and F.

\noindent \textbf{Case N: $\gamma \in \omega$.}  For $\gamma \in \omega$ we let $\Gamma_{\gamma} = \Bo^{\gamma} = \Sequ_{\gamma} = I_{\Log, \gamma} = I_{\Rel, \gamma} = I_{\Conn, \gamma} = \emptyset$.  It is easy to see that (i) - (vi) hold at each such $\gamma$.

For the remaining cases we suppose that (i) - (vi) holds for all $\beta <\gamma$, where $\omega \leq \gamma <\aleph_1$.

\noindent \textbf{Case L: $\gamma$ is a limit ordinal.}  Suppose $\gamma$ is a limit ordinal.  We let $I_{\Log, \gamma} = \bigcup_{\beta < \gamma} I_{\Log, \beta}$ and $\Gamma_{\gamma} = \bigcup_{\beta < \gamma} \Gamma_{\beta}$.  Let $I_{\Rel, \gamma} = \bigcup_{\beta <\gamma}  I_{\Rel, \beta}$.  For each $\alpha \in I_{\Rel, \gamma}$ we let $B_{\alpha}^{\gamma} = \bigcup_{\beta < \gamma, \alpha \in I_{\Rel, \beta}} B_{\alpha}^{\beta}$, and write $\Bo^{\gamma} = \{B_{\alpha}^{\gamma}\}_{\alpha \in I_{\Rel, \gamma}}$.  Let $I_{\Conn, \gamma} = \bigcup_{\beta < \gamma} I_{\Conn, \beta}$ and $\Sequ_{\gamma} = \bigcup_{\beta < \gamma} \Sequ_{\beta}$.  We'll check that conditions (i) - (vi) are satisfied.

Certainly (i) and (ii) are clear.  To see (iii), we let $\alpha \in I_{\Rel, \gamma}$ be given.  We have $\alpha \in I_{\Rel, \beta}$ for some $\beta < \gamma$ and since $\Bo^{\beta} \models \B_{\alpha} \neq \emptyset$, we have $\emptyset \neq B_{\alpha}^{\beta} = B_{\alpha}^{\gamma}\cap \beta$, so that in particular $\Bo^{\gamma}\models \B_{\alpha} \neq \emptyset$.  Now if $\Bo^{\gamma} \models \Gamma_{\gamma}$ fails, say $\Bo^{\gamma} \models \Theta_{\alpha}$ fails, we know by Remark \ref{existentialproblems} that there is an existential witness for this.  For example if $\Theta_{\alpha}$ is of type (a) then we have some $x\in \gamma$ such that $x\in B_{\alpha_0}^{\gamma} \cap B_{\alpha_1}^{\gamma}$, but now we can select $\beta < \gamma$ with $x\in B_{\alpha_0}^{\beta} \cap B_{\alpha_1}^{\beta}$ and $\alpha \in I_{\Log, \beta}$ which contradicts $\Bo^{\beta} \models \Theta_{\alpha}$.  If $\Theta_{\alpha}$ is of type (b), (c), or (d) then the check is similar.  Thus (iii) holds.

For (iv) we consider one of two possibilities.  If $\gamma = \omega$ then $(\omega, \omega \cap\{B_{\alpha}^{\gamma}\}_{\alpha \in I_{\Rel, \gamma}}) = (\omega, \emptyset)$ and this is certainly a Boolean saturated model of $\Gamma_{\omega} = \emptyset$.  Suppose now that $\gamma > \omega$.  By (iii) we know that $\Gamma_{\gamma}$ is consistent.  The check that $(\omega, \omega \cap\{B_{\alpha}^{\gamma}\}_{\alpha \in I_{\Rel, \gamma}})$ is a model of $\Gamma_{\gamma}$ follows along precisely the same lines as the check in (iii).  We check that this model is Boolean saturated.  Let $Y, Z \subseteq I_{\Rel, \gamma}$ be finite such that $\X_{\gamma, Y, Z} \neq \emptyset$.  Select $\omega < \beta < \gamma$ large enough that $Y, Z \subseteq I_{\Rel, \beta}$.  Since $(\omega, \{\omega \cap B_{\alpha}^{\beta}\}_{\alpha \in I_{\Rel, \beta}})$ is a Boolean saturated model of $\Gamma_{\beta}$, we use the coordinate map $\Coor_{\beta}: \omega \rightarrow \X_{\beta}$ and have by assumption that $\Coor_{\beta}^{-1}(\X_{\beta, Y, Z})$ is infinite.  But letting $\Coor_{\gamma}: \omega \rightarrow \X_{\gamma}$ be the $\gamma$-coordinate map, it is clear that $\Coor_{\gamma}^{-1}(\X_{\gamma, Y, Z}) \supseteq \Coor_{\beta}^{-1}(\X_{\beta, Y, Z})$.  We have shown (iv).

The claim (v) is clear by induction.

For (vi) we let $\alpha_0 \in I_{\Conn, \gamma}$ and $\alpha_1 \in I_{\Rel, \gamma}$ and suppose that $\alpha_0 \in B_{\alpha_1}^{\gamma}$.  As $\gamma$ is a limit ordinal we select $\alpha_0, \alpha_1 < \beta < \gamma$ such that $\alpha_0 \in I_{\Conn, \beta}$ and $\alpha_1 \in I_{\Rel, \beta}$.  We have $\alpha_0 \in B_{\alpha_1}^{\gamma} \cap \beta = B_{\alpha_1}^{\beta}$, so by inductive hypothesis there exists $N\in \omega$ such that $\{s_{\alpha_0, 0, q}\}_{q \geq N} \cup \{s_{\alpha_0, 1, q}\}_{q \geq N} \subseteq B_{\alpha_1}^{\beta} \subseteq B_{\alpha_1}^{\gamma}$.  This completes Case L.

\noindent \textbf{Case A: $\gamma = t + \omega(6k + 2) + n, n \neq 0$.}  Let $\beta_0 = t + \omega(6k + 2)$.  By assumption the conditions (i) - (vi) hold at and below $\beta_0$.  We will take care of these $\gamma = \beta_0 + n$ by induction on $n \geq 1$, dealing with $n = 2p + 1$ and $n = 2p + 2$ at the same time.  Thus we will be assuming that $n\in 2\omega + 1$ and the successor of such an $n$ will be considered at the same time.

For all $\gamma \in [\beta_0, \beta_0 + \omega)$ we will let

\begin{center}

$I_{\Conn, \gamma} = I_{\Conn, \beta_0}$

\end{center}

\noindent and

\begin{center}

$\Sequ_{\gamma} = \Sequ_{\beta_0}$.

\end{center}

Let $\{(x_m, y_m)\}_{m \in 2\omega}$ be an enumeration of all ordered pairs of distinct points $x, y \in \beta_0$.  We have $n \in 2\omega +1$ and we let $m = n - 1$.

Let $\Theta_{\beta_0 + n}  \equiv [\B_{\beta_0 + n - 1} \cap \B_{\beta_0 + n} = \emptyset] \wedge [\B_{\beta_0 + n - 1} \neq \emptyset]\wedge[\B_{\beta_0 + n} \neq \emptyset]$.  We note by Lemma \ref{nextstagea} that $\Gamma_{\beta_0 + m} \cup \{\Theta_{\beta_0 + n}\}$ is consistent.  We are assuming by induction that (i) - (vi) hold at and below $\beta_0 + m$.

We select distinct points $x, y \in \omega$, where $x = x_m$ if $x_m \in \omega$ and similarly for $y$ and $y_m$.  By Lemma \ref{refinecasea} together with induction condition (iv) there exist $J_{\beta_0 + m}, J_{\beta_0 + m + 1} \subseteq \omega$ with $x \in J_{\beta_0 + m}$ and $y \in J_{\beta_0 + m + 1}$ such that $$(\omega, \{\omega \cap B_{\alpha}^{\beta_0 + m}\}_{\alpha \in I_{\Rel, \beta_0 + m}} \cup \{J_{\beta_0 + m}, J_{\beta_0 + m + 1}\})$$ is a Boolean saturated model of $\Gamma_{\beta_0 + m} \cup \{\Theta_{\beta_0 + n}\}$.

\noindent \textbf{Subcase A.1: $x_m, y_m \in \beta_0 \setminus I_{\Conn, \beta}$.}  Notice that $$(\beta_0 + m, \{B_{\alpha}^{\beta_0 + m}\}_{\alpha \in I_{\Rel, \beta_0 + m}} \cup \{J_{\beta_0 + m} \cup \{x_m\}, J_{\beta_0 + m + 1} \cup\{y_m\}\})$$  is a model of $\Gamma_{\beta_0 + m} \cup \{\Theta_{\beta_0 + n}\}$.  We let $B_{\beta_0 + m}^{\beta_0 + m} = J_{\beta_0 + m} \cup \{x_m\}$ and for $\alpha \in I_{\Rel, \beta_0 + m} \cup \{\beta_0 + m, \beta_0 + m + 1\}$ we let

\begin{center}

$B_{\alpha}^{\beta_0 + m + 1} = \begin{cases}B_{\alpha}^{\beta_0 + m}$ if $\alpha < \beta_0 + m$ and $x_m \notin B_{\alpha}^{\beta_0 + m}\\ B_{\alpha}^{\beta_0 + m} \cup \{\beta_0 + m\}$ if $\alpha < \beta_0 + m$ and $x_m\in B_{\alpha}^{\beta_0 + m}\\ J_{\beta_0 + m + 1} \cup \{y_m\}$ if $\alpha = \beta_0 +  m + 1\end{cases}$

\end{center}

\noindent and 

\begin{center}  $B_{\alpha}^{\beta_0 + m + 2} = \begin{cases}B_{\alpha}^{\beta_0 + m + 1}$ if $y_m \notin B_{\alpha}^{\beta_0 + m + 1}\\B_{\alpha}^{\beta_0 + m + 1} \cup \{\beta_0 + m + 1\}$ if $y_m \in B_{\alpha}^{\beta_0 + m + 1}   \end{cases}$
\end{center}

We let $I_{\Log, \beta_0 + m + 2} = I_{\Log, \beta_0 + m} \cup \{\beta_0 + m + 1\}$ and $\Gamma_{\beta_0 + m + 2} = \Gamma_{\beta_0 + m} \cup \{\Theta_{\beta_0 +  m + 1}\}$, and $I_{\Rel, \beta_0  + m + 2} = I_{\Rel, \beta_0 + m} \cup \{\beta_0 + m, \beta_0 + m + 1\}$.  By Lemma \ref{addingpoints} we have $(\beta + m + 2, \{B_{\alpha}^{\beta_0 + m + 2}\}_{\alpha \in I_{\Rel, \beta_0 + m + 2}}) \models \Gamma_{\beta_0 + m + 2}$.  The check that conditions (i), (ii), (iv), (v) and (vi) also hold at $\gamma = \beta_0 + m + 1 = \beta_0 + n$ and $\gamma = \beta_0 + m + 2 = \beta_0 + n + 1$ hold is quite obvious.

\noindent \textbf{Subcase A.2: $x_m, y_m \in I_{\Conn, \beta_0}$.}  We have $I_{\Conn, \beta_0} \subseteq T$ and $\{s_{x_m, 0, q}\}_{q\in \omega} \cup \{s_{x_m, 1, q}\}_{q\in \omega} \subseteq \beta_0 \setminus (T \cup \omega)$, and similarly for $y_m$.  If without loss of generality $x_m > y_m$ (considered as elements of $\beta_0$) we select $N \in \omega$ large enough that $q \geq N$ implies $x_{x_m, i, q} > y_m$ for $i \in \{0, 1\}$.   Notice that $(\beta_0 + m, \{B_{\alpha}^{\beta_0 + m}\}_{\alpha \in I_{\Rel, \beta_0 + m}} \cup \{J_{\beta_0 + m} \cup \{x_m\} \cup \{s_{x_m, i, q}\}_{i\in \{0, 1\}, q \geq N}, J_{\beta_0 + m + 1} \cup \{y_m\} \cup \{s_{y_m, i, q}\}_{i\in \{0, 1\}, q\in \omega}\})$ is a model of  $\Gamma_{\beta_0 + m} \cup \{\Theta_{\beta_0 + n}\}$.  Now we let $B_{\beta_0 + m}^{\beta_0 + m} = J_{\beta_0 + m} \cup \{x_m\} \cup \{s_{x_m, i, q}\}_{i\in \{0, 1\}, q \geq N}$ and for $\alpha \in I_{\Rel, \beta_0 + n} \cup \{\beta_0 + m, \beta_0 + m +1\}$ we let

\begin{center}

$B_{\alpha}^{\beta_0 + m + 1} = \begin{cases}B_{\alpha}^{\beta_0 + m}$ if $\alpha < \beta + m + 1$ and $x_m \notin B_{\alpha}^{\beta_0 + m}\\B_{\alpha}^{\beta_0 + m} \cup \{\beta_0 + m\}$ if $\alpha < \beta_0 + m + 1$ and  $x_m\in B_{\alpha}^{\beta_0 + m}\\J_{\beta_0 + m + 1} \cup \{y_m\} \cup \{s_{y_m, i, q}\}_{i\in \{0, 1\}, q\in \omega}\})$ if $\alpha = \beta_0 + m + 1   \end{cases}$

\end{center}

\noindent and

\begin{center}  $B_{\alpha}^{\beta_0 + m + 2} = \begin{cases}B_{\alpha}^{\beta_0 + m + 1}$ if $y_m \notin B_{\alpha}^{\beta_0 + m + 1}\\B_{\alpha}^{\beta_0 + m + 1} \cup \{\beta_0 + m + 1\}$ if $y_m \in B_{\alpha}^{\beta_0 + m + 1}   \end{cases}$
\end{center}

We define $I_{\Log, \beta_0 + m + 2}$, $\Gamma_{\beta_0 + m + 2}$, and $I_{\Rel, \beta_0 + m + 2}$ as in Subcase A.1.  That (i) - (vi) hold at $\gamma = \beta_0 + m + 1$ and $\gamma = \beta_0 + m + 2$ is clear.

\noindent \textbf{Subcase A.3: Exactly one of $x_m, y_m$ is in $I_{\Conn, \beta_0}$.}  Without loss of generality $x_m \in I_{\Conn, \beta}$.  Select $N \in \omega$ large enough that $y_m \notin \{x_m\} \cup \{s_{x_m, i, q}\}_{i\in \{0, 1\}, q \geq N}$ (if $x_m < y_m$ then of course we can let $N = 0$ and if $x_m > y_m$ then we are using the fact that $x_m = \lim_{q \rightarrow \infty} s_{x_m, i, q}$).

Notice that $$(\beta_0 + m, \{B_{\alpha}^{\beta_0 + m}\}_{\alpha \in I_{\Rel, \beta_0 + m}} \cup \{J_{\beta_0 + m} \cup \{x_m\} \cup \{s_{x_m, i, q}\}_{i\in \{0, 1\}, q \geq N}, J_{\beta_0 + m + 1} \cup\{y_m\}\})$$ is a model of  $\Gamma_{\beta_0 + m} \cup \{\Theta_{\beta_0 + n}\}$.  Let $B_{\beta_0 + m}^{\beta_0 + m} = J_{\beta_0 + m} \cup \{x_m\} \cup \{s_{x_m, i, q}\}_{i\in \{0, 1\}, q \geq N}$ and for $\alpha \in I_{\Rel, \beta_0 + m} \cup \{\beta_0 + m, \beta_0 + m +1\}$ we let

\begin{center}

$B_{\alpha}^{\beta_0 + m + 1} = \begin{cases}B_{\alpha}^{\beta_0 + m}$ if $\alpha < \beta + m + 1$ and $x_m \notin B_{\alpha}^{\beta_0 + m}\\B_{\alpha}^{\beta_0 + m} \cup \{\beta_0 + m\}$ if $\alpha < \beta_0 + n + 1$ and $x_m\in B_{\alpha}^{\beta_0 + m}\\J_{\beta_0 + m + 1} \cup \{y_m\}$ if $\alpha = \beta_0 + m + 1   \end{cases}$

\end{center}

\noindent and

\begin{center}  $B_{\alpha}^{\beta_0 + m + 2} = \begin{cases}B_{\alpha}^{\beta_0 + m + 1}$ if $y_m \notin B_{\alpha}^{\beta_0 + m + 1}\\B_{\alpha}^{\beta_0 + m + 1} \cup \{\beta_0 + m + 1\}$ if $y_m \in B_{\alpha}^{\beta_0 + m + 1}   \end{cases}$
\end{center}

Update the parameters as in Subcase A.1 and (i) - (vi) follow.

\noindent \textbf{Case B: $ \gamma = t + \omega(6k + 3) + n, n \neq 0$.}  Let $\beta_0 = t+ \omega(6k + 3)$.  We know that conditions (i) - (vi) hold at and below $\beta_0$.  We will induct on $n \geq 1$.  For all $\gamma$ in this case we will have 

\begin{center}
$I_{\Conn, \gamma} = I_{\Conn, \beta_0}$
\end{center}

\noindent and

\begin{center}
$\Sequ_{\gamma} = \Sequ_{\beta_0}$
\end{center}

Let $\{(B_{\alpha_{m, 0}}^{\beta_0}, B_{\alpha_{m, 1}}^{\beta_0}, x_m)\}_{m \in \omega}$ be an enumeration of all triples such that $x_m, \alpha_{m, 0}, \alpha_{m, 1} < \beta_0$ and $x_m \in B_{\alpha_{m, 0}}^{\beta_0} \cap B_{\alpha_{m, 1}}^{\beta_0}$.  Let $m = n - 1$ and $\Theta_{\beta_0 + m} \equiv [\B_{\alpha_{m, 0}} \cap \B_{\alpha_{m, 1}} \supseteq \B_{\beta_0 + m}]\wedge [\B_{\beta_0 + m} \neq \emptyset]$.  We note by Lemma \ref{nextstageb} that $\Gamma_{\beta_0 + m} \cup \{\Theta_{\beta_0 + m}\}$ is consistent, since $\emptyset \neq \Coor_{\beta_0 + m}(B_{\alpha_{m, 0}}^{\beta_0 + m}) \cap \Coor_{\beta_0 + m}(B_{\alpha_{m, 1}}^{\beta_0 + m}) \subseteq b_{\alpha_{m, 0}}^{\beta_0 + m} \cap b_{\alpha_{m, 1}}^{\beta_0 + m}$.

We have by induction that $(\omega, \{\omega \cap B_{\alpha}^{\beta_0 + m}\}_{\alpha \in I_{\Rel, \beta_0 + m}})$ is a Boolean saturated model of $\Gamma_{\beta_0 + m}$ and we select $x \in \omega \cap B_{\alpha_{m, 0}}^{\beta_0} \cap B_{\alpha_{n, 1}}^{\beta_0} = \omega \cap B_{\alpha_{m, 0}}^{\beta_0 + m} \cap B_{\alpha_{m, 1}}^{\beta_0 + m}$, with $x = x_m$ in case $x_m \in \omega$.  By Lemma \ref{refinecaseb} we let $J_{\beta_0 + m} \subseteq \omega$, with $x\in J_{\beta_0 + m}$, be such that $$(\omega, \{\omega \cap B_{\alpha}^{\beta_0 + m}\}_{\alpha \in I_{\Rel, \beta_0 + m}} \cup\{J_{\beta_0 + m}\})$$ is a Boolean saturated model of $\Gamma_{\beta_0 + m} \cup \{\Theta_{\beta_0 + m}\}$.

\noindent \textbf{Subcase B.1: $x_m \in \beta_0 \setminus I_{\Conn, \beta_0}$.}  It is clear that $(\beta_0 + m, \{B_{\alpha}^{\beta_0 + m}\}_{\alpha \in I_{\Rel, \beta_0 + m}} \cup \{J_{\beta_0 + m} \cup \{x_m\}\})$ is a model of $\Gamma_{\beta_0 + m} \cup \{\Theta_{\beta_0 + m}\}$.  We let $B_{\beta_0 + m}^{\beta_0 + m} = J_{\beta_0 + m} \cup \{x_m\}$ and for $\alpha \in I_{\Rel, \beta_0 + m} \cup \{\beta_0 + m\}$ we let

\begin{center}  $B_{\alpha}^{\beta_0 + m + 1} = \begin{cases} B_{\alpha}^{\beta_0 + m}$ if $x_m \notin B_{\alpha}^{\beta_0 + m}\\  B_{\alpha}^{\beta_0 + m} \cup \{\beta_0 + m\}$ if $x_m \in B_{\alpha}^{\beta_0 + m}  \end{cases}$
\end{center}

Let $I_{\Log, \beta_0 + m +1} = I_{\Log, \beta_0 + m} \cup \{\beta_0 + m\}$, $\Gamma_{\beta_0 + m + 1} = \Gamma_{\beta_0 + m} \cup \{\Theta_{\beta_0 + m}\}$, and $I_{\Rel, \beta_0 + m + 1} = I_{\Rel, \beta_0 + m} \cup \{\beta_0 + m\}$.  By Lemma \ref{addingpoints} we have that $\Bo^{\beta_0 + m + 1} \models \Gamma_{\beta_0 + m + 1}$, and the checks for the remaining (i), (ii), (iv), (v), and (vi) are straightforward.

\noindent \textbf{Subcase B.2: $x_m \in I_{\Conn, \beta_0}$.}  As property (vi) holds at $\beta_0 + m$ we select $N \in \omega$ large enough that both $\{s_{x_m, i, q}\}_{i\in \{0, 1\}, q \geq N} \subseteq B_{\alpha_{m, 0}}$ and $\{s_{x_m, i, q}\}_{i\in \{0, 1\}, q \geq N} \subseteq B_{\alpha_{m, 1}}$.

It is clear that $(\beta_0 + m, \{B_{\alpha}^{\beta_0 + m}\}_{\alpha \in I_{\Rel, \beta_0 + m}} \cup \{J_{\beta_0 + m} \cup \{x_m\} \cup \{s_{x_m, i, q}\}_{i\in \{0, 1\}, q \geq N}\})$ is a model of $\Gamma_{\beta_0 + m} \cup \{\Theta_{\beta_0 + m}\}$.  We let $B_{\beta_0 + m}^{\beta_0 + m} = J_{\beta_0 + m} \cup \{x_m\} \cup \{s_{x_m, i, q}\}_{i\in \{0, 1\}, q \geq N}$ and for $\alpha \in I_{\Rel, \beta_0 + m} \cup \{\beta_0 + m\}$ we let

\begin{center}  $B_{\alpha}^{\beta_0 + m + 1} = \begin{cases} B_{\alpha}^{\beta_0 + m}$ if $x_m \notin B_{\alpha}^{\beta_0 + m}\\  B_{\alpha}^{\beta_0 + m} \cup \{\beta_0 + m\}$ if $x_m \in B_{\alpha}^{\beta_0 + m}  \end{cases}$
\end{center}

One updates the parameters as in the end of Subcase B.1.  That (i) - (vi) hold at $\gamma = \beta_0 + m + 1 = \beta_0 + n$ is easy to see.

\noindent \textbf{Case C: $ \gamma = t + \omega(6k + 4) + n, n \neq 0$.}  Let $\beta_0 = t + \omega(6k + 4)$.  As in Case A we will deal with these $\gamma$ by induction on $n \geq 1$, and consider $n = 2p + 1$ and $n = 2p + 2$ contemporaneously.  Thus, as in Case A, we will assume that $n\in 2\omega + 1$.  For all of the $\gamma$ being considered in this case we let

\begin{center}
$I_{\Conn, \gamma} = I_{\Conn, \beta_0}$
\end{center}

\noindent and

\begin{center}
$\Sequ_{\gamma} = \Sequ_{\beta_0}$
\end{center}

Let $\{(B_{\alpha_m}^{\beta_0}, x_m, y_m)\}_{m \in 2\omega}$ be an enumeration of all triples such that $\alpha_m \in I_{\Rel, \beta_0}$, $x_m, y_m \in \beta_0$, $x_m \neq y_m$, and $x_m, y_m \notin B_{\alpha_m}^{\beta_0}$.  We have $n\in 2\omega + 1$ and let $m = n - 1$.

Let $\Theta_{\beta_0 + n}\equiv [\B_{\beta_0 + m} \cap \B_{\beta_0 + m + 1} \subseteq \B_{\alpha_m}] \wedge (\forall x)[x \in \B_{\beta_0 + m}\cup \B_{\beta_0 + m + 1}\cup \B_{\alpha_m}] \wedge [\B_{\beta_0 + m} \neq \emptyset] \wedge [\B_{\beta_0 + m + 1} \neq \emptyset]$.  We note by Lemma \ref{nextstagec} that $\Gamma_{\beta_0 + m} \cup \{\Theta_{\beta_0 + n}\}$ is consistent.

We know by induction that $(\omega, \{\omega \cap B_{\alpha}^{\beta_0 + m}\}_{\alpha \in I_{\Rel, \beta_0 + m}})$ is a Boolean saturated model of $\Gamma_{\beta_0 + m}$.  Select distinct points $x, y\in \omega \setminus B_{\alpha_m}^{\beta_0 + m}$ such that $x = x_k$ or $y = y_k$ if $x_k$ or $y_k$ is in $\omega$.  That such points exist is implied by the fact that $(\omega, \{\omega \cap B_{\alpha}^{\beta_0 + m}\}_{\alpha \in I_{\Rel, \beta_0 + m}})$ is Boolean saturated, for $\Coor_{\beta_0 + m}(\{x_m, y_m\}) \subseteq \X_{\beta_0 + m, \{\alpha_m\}, \emptyset}$.  By Lemma \ref{refinecasec} there exist $J_{\beta_0 + m}, J_{\beta_0 + m + 1} \subseteq \omega$ such that $x\in J_{\beta_0 + m}$, $y \in J_{\beta_0 + m + 1}$ and

\begin{center}

$(\omega, \{\omega \cap B_{\alpha}^{\beta_0 + m}\}_{\alpha \in I_{\Rel, \beta_0 + m}} \cup \{J_{\beta_0 + m}, J_{\beta_0 + m + 1}\})$

\end{center}

\noindent is a Boolean saturated model of $\Gamma_{\beta_0 + m} \cup \{\Theta_{\beta_0 + n}\}$.

\noindent \textbf{Subcase C.1: $x_m, y_m \in \beta_0 \setminus I_{\Conn, \beta_0}$.} Notice that 

\begin{center}

$(\beta_0 + m, \{B_{\alpha}^{\beta_0 + m}\}_{\alpha \in I_{\Rel, \beta_0 + m}} \cup \{J_{\beta_0 + m} \cup \{x_m\}, J_{\beta_0 + m + 1} \cup ((\beta_0 + m)\setminus (\omega\cup \{x_m\}))\})$

\end{center}

\noindent is a model of $\Gamma_{\beta_0 + m} \cup \{\Theta_{\beta_0 + n}\}$.  We let $B_{\beta_0 + m}^{\beta_0 + m} = J_{\beta_0 + n} \cup \{x_m\}$ and for $\alpha \in I_{\Rel, \beta_0 + m}\cup\{\beta_0 + m, \beta_0 + m + 1\}$ we let

\begin{center}

$B_{\alpha}^{\beta_0 + m + 1} = \begin{cases} B_{\alpha}^{\beta_0 + m}$ if $\alpha < \beta_0 + m + 1$ and $x_m \notin B_{\alpha}^{\beta_0 + m}\\B_{\alpha}^{\beta_0 + m} \cup \{\beta_0 + m\}$ if $\alpha < \beta_0 + m + 1$ and $x_m \in B_{\alpha}^{\beta_0 + m}\\J_{\beta_0 + m + 1} \cup ((\beta_0 + m)\setminus(\omega\cup \{x_m\}))$ if $\alpha = \beta_0 + m + 1  \end{cases}$

\end{center}

\noindent and

\begin{center}

$B_{\alpha}^{\beta_0 + m + 2} = \begin{cases}B_{\alpha}^{\beta_0 + m + 1}$ if $y_m \notin B_{\alpha}^{\beta_0 + m + 1}\\B_{\alpha}^{\beta_0 + m + 1} \cup \{\beta_0 + m + 1\}$ if $y_m \in B_{\alpha}^{\beta_0 + m + 1}   \end{cases}$

\end{center}

Let $I_{\Log, \beta_0 + m + 2} = I_{\Log, \beta_0 + m} \cup \{\beta_0 +  n\}$ and $\Gamma_{\beta_0 + m + 2} = \Gamma_{\beta_0 + m} \cup \{\Theta_{\beta_0 + n}\}$, and $I_{\Rel, \beta_0 + m + 2} = I_{\Rel, \beta_0 + m} \cup \{\beta_0 + m, \beta_0 + m + 1\}$.  By Lemma \ref{addingpoints} we have that $\Bo^{\beta_0 + m + 2} \models \Gamma_{\beta_0 + m + 2}$, and checking that the other induction hypotheses at $\beta_0 + m + 1$ and $\beta_0 + m + 2$ are met is straightforward.

\noindent \textbf{Subcase C.2: At least one of $x_m, y_m$ is in $I_{\Conn, \beta_0}$.}  Let, without loss of generality, $x_m \in I_{\Conn, \beta_0}$.  We have $I_{\Conn, \beta_0} \subseteq T$ and $\{s_{x_m, i, q}\}_{i\in \{0, 1\}, q\in \omega} \subseteq \beta_0 \setminus (T \cup \omega)$.  As $x_m, y_m$ are distinct we select $N \in \omega$ large enough that $y_m \notin [x_{x_m, N, 0}, x_m) \cup [x_{x_m, N, 0}, x_m)$.

Notice that 

\begin{center}

$(\beta_0 + m, \{B_{\alpha}^{\beta_0 + m}\}_{\alpha \in I_{\Rel, \beta_0 + m}} \cup \{J_{\beta_0 + m} \cup \{x_m\} \cup \{s_{x_m, i, q}\}_{i \in \{0, 1\}, q \geq N}, J_{\beta_0 + m + 1} \cup ((\beta_0 + m) \setminus (\omega \cup  \{x_m\} \cup \{s_{x_m, i, q}\}_{i \in \{0, 1\}, q \geq N}))\})$

\end{center}

\noindent is a model of $\Gamma_{\beta_0 + m}\cup\{\Theta_{\beta_0 + n}\}$.  We let $B_{\beta_0 + m}^{\beta_0 + m} = J_{\beta_0 + n} \cup\{x_m\} \cup \{s_{x_m, i, q}\}_{i \in \{0, 1\}, q \geq N}$ and for $\alpha \in I_{\Rel, \beta_0 + m}\cup\{\beta_0 + m, \beta_0 + m + 1\}$ we let $B_{\alpha}^{\beta_0 + m + 1}$ equal

\begin{center}

$\begin{cases} B_{\alpha}^{\beta_0 + m}$ if $\alpha < \beta_0 + m + 1$ and $x_m \notin B_{\alpha}^{\beta_0 + m}\\B_{\alpha}^{\beta_0 + m} \cup \{\beta_0 + m\}$ if $\alpha < \beta_0 + m + 1$ and $x_m \in B_{\alpha}^{\beta_0 + m}\\J_{\beta_0 + m + 1} \cup ((\beta_0 + m)\setminus (\omega\cup \{x_k\} \cup \{s_{x_m, i, q}\}_{i \in \{0, 1\}, q\geq N}))$ if $\alpha = \beta_0 + m + 1  \end{cases}$

\end{center}

\noindent and

\begin{center}

$B_{\alpha}^{\beta_0 + m + 2} = \begin{cases}B_{\alpha}^{\beta_0 + m + 1}$ if $y_m \notin B_{\alpha}^{\beta_0 + m + 1}\\B_{\alpha}^{\beta_0 + m + 1} \cup \{\beta_0 + m + 1\}$ if $y_m \in B_{\alpha}^{\beta_0 + m + 1}   \end{cases}$

\end{center}

Update the parameters as in Subcase C.1 and the check that the inductive conditions hold is straightforward.

\noindent \textbf{Case D: $ \gamma = t + \omega(6k + 5) + n, n \neq 0$.}  Let $\beta_0 = t + \omega(6k + 5)$.  As in Cases A and C we will induct on $n \geq 1$ and consider $n = 2p + 1$ and $n = 2p + 2$ together.  Thus, as in those cases, we will assume that $n\in 2\omega + 1$.  For all of the $\gamma$ being considered in this case we let

\begin{center}
$I_{\Conn, \gamma} = I_{\Conn, \beta_0}$
\end{center}

\noindent and

\begin{center}
$\Sequ_{\gamma} = \Sequ_{\beta_0}$
\end{center}

Let $\{(B_{\alpha_m}^{\beta_0}, x_m)\}_{m \in 2\omega}$ be an enumeration of all pairs  where $x_m \in B_{\alpha_m}^{\beta_0}$.  We have $n\in 2\omega + 1$ and let $m = n - 1$.

Let $\Theta_{\beta_0 + n} \equiv (\forall x)[x\in \B_{\alpha_m} \cup \B_{\beta_0 + m + 1}] \wedge [\B_{\beta_0 + m} \subseteq \B_{\alpha_m}] \wedge [\B_{\beta_0 + m} \neq \emptyset] \wedge [\B_{\beta_0 + m + 1}\neq \emptyset]$.  We note by Lemma \ref{nextstaged} that $\Gamma_{\beta_0 + n} \cup \{\Theta_{\beta_0 + n + 1}\}$ is consistent.

We know $(\omega, \{\omega \cap B_{\alpha}^{\beta_0 + m}\}_{\alpha \in I_{\Rel, \beta_0 + m}})$ is a Boolean saturated model of $\Gamma_{\beta_0 + m}$.  Select a point $x\in \omega \cap B_{\alpha_m}^{\beta_0 + m}$, with $x = x_m$ if $x_m \in \omega$.  By Lemma \ref{refinecased} we can select $J_{\beta_0 + m}, J_{\beta_0 + m + 1} \subseteq \omega$ with $x\in J_{\beta_0 + m}$ and

\begin{center}

$(\omega, \{\omega \cap B_{\alpha}^{\beta_0 + m}\}_{\alpha \in I_{\Rel, \beta_0 + m}} \cup \{J_{\beta_0 + m}, J_{\beta_0 + m + 1}\})$

\end{center}

\noindent is a Boolean saturated model of $\Gamma_{\beta_0 + m} \cup \{\Theta_{\beta_0 + n}\}$.

\noindent \textbf{Subcase D.1: $x_m \in \beta_0 \setminus I_{\Conn, \beta_0}$.}  Notice that

\begin{center}

$(\beta_0 + m, \{B_{\alpha}^{\beta_0 + m}\}_{\alpha \in I_{\Rel, \beta_0 + m}} \cup  \{\{x_m\} \cup J_{\beta_0 + m}, J_{\beta_0 + m + 1}\cup ((\beta_0 + m) \setminus (\omega \cup \{x_m\}))\})$

\end{center}

\noindent is a model of $\Gamma_{\beta_0 + m} \cup \{\Theta_{\beta_0 + n}\}$.  We let $B_{\beta_0 + m}^{\beta_0 + m} = J_{\beta_0 + m} \cup \{x_m\}$ and for $\alpha \in I_{\Rel, \beta_0 + m} \cup \{\beta_0 + m, \beta_0 + m + 1\}$ we let

\begin{center}

$B_{\alpha}^{\beta_0 + m + 1} = \begin{cases} B_{\alpha}^{\beta_0 + m}$ if $\alpha < \beta_0 + m + 1$ and $x_m\notin B_{\alpha}^{\beta_0 + m}\\B_{\alpha}^{\beta_0 + m} \cup \{\beta_0 + m\}$ if $\alpha < \beta_0 + m + 1$ and $x_m \in B_{\alpha}^{\beta_0 + m}\\J_{\beta_0 + m + 1} \cup ((\beta_0 + m)\setminus (\omega\cup \{x_k\}))$ if $\alpha = \beta_0 + m + 1  \end{cases}$

\end{center}

\noindent and

\begin{center}

$B_{\alpha}^{\beta_0 + m + 2} = \begin{cases}B_{\alpha}^{\beta_0 + m + 1}$ if $x_m \notin B_{\alpha}^{\beta_0 + m + 1}\\B_{\alpha}^{\beta_0 +  + 1} \cup \{\beta_0 + m + 1\}$ if $x_m \in B_{\alpha}^{\beta_0 + m + 1}   \end{cases}$

\end{center}

Let $I_{\Log, \beta_0 + m + 2} = I_{\Log, \beta_0 + m} \cup \{\beta_0 +  n\}$ and $\Gamma_{\beta_0 + m + 2} = \Gamma_{\beta_0 + m} \cup \{\Theta_{\beta_0 + n}\}$, and $I_{\Rel, \beta_0 + m + 2} = I_{\Rel, \beta_0 + m} \cup \{\beta_0 + m, \beta_0 + m + 1\}$.  By Lemma \ref{addingpoints} we have that $\Bo^{\beta_0 + m + 2} \models \Gamma_{\beta_0 + m + 2}$, and checking that the other induction hypotheses are met is straightforward.

\noindent \textbf{Subcase D.2: $x_m \in I_{\Conn, \beta_0}$.}  By induction we know that we can select $N \in \omega$ such that $\{s_{x_m, i, q}\}_{i\in \{0, 1\}, q \geq N} \subseteq B_{\alpha_m}$.  Notice that

\begin{center}

$(\beta_0 + m, \{B_{\alpha}^{\beta_0 + m}\}_{\alpha \in I_{\Rel, \beta_0 + m}} \cup  \{\{x_m\} \cup J_{\beta_0 + m} \{s_{x_m, i, q}\}_{i\in \{0, 1\}, q \geq N}, J_{\beta_0 + m + 1}\cup ((\beta_0 + m) \setminus (\omega \cup \{x_m\} \cup \{s_{x_m, i, q}\}_{i\in \{0, 1\}, q \geq N}))\})$

\end{center}

\noindent is a model of $\Gamma_{\beta_0 + m} \cup \{\Theta_{\beta_0 + n}\}$.  We let $B_{\beta_0 + m}^{\beta_0 + m} = J_{\beta_0 + m}\cup \{x_m\} \cup \{s_{x_m, i, q}\}_{i\in \{0, 1\}, q \geq N}$ and for $\alpha \in I_{\Rel, \beta_0 + m} \cup \{\beta_0 + m, \beta_0 + m + 1\}$ we let $B_{\alpha}^{\beta_0 + m + 1}$ equal

\begin{center}

$\begin{cases} B_{\alpha}^{\beta_0 + m}$ if $\alpha < \beta_0 + m + 1$ and $x_m \notin B_{\alpha}^{\beta_0 + m}\\B_{\alpha}^{\beta_0 + m} \cup \{\beta_0 + m\}$ if $\alpha < \beta_0 + m + 1$ and $x_m \in B_{\alpha}^{\beta_0 + m}\\J_{\beta_0 + m + 1} \cup ((\beta_0 + m) \setminus (\omega \cup \{x_m\} \cup\{s_{x_m, i, q}\}_{i\in \{0, 1\}, q \geq N}))$ if $\alpha = \beta_0 + m + 1  \end{cases}$

\end{center}

\noindent and

\begin{center}

$B_{\alpha}^{\beta_0 + m + 2} = \begin{cases}B_{\alpha}^{\beta_0 + m + 1}$ if $x_m \notin B_{\alpha}^{\beta_0 + m + 1}\\B_{\alpha}^{\beta_0 + m + 1} \cup \{\beta_0 + m + 1\}$ if $x_m \in B_{\alpha}^{\beta_0 + m + 1}   \end{cases}$

\end{center}

Update the other parameters as in Subcase D.1.  Again, one can check that the inductive conditions are met for $\gamma = \beta_0 + m + 1$ and $\gamma = \beta_0 + m + 2$.

\noindent \textbf{Case E: $\gamma = t + \omega(6k + 1) + n, n \neq 0$.}  Let $\beta_0 = t + \omega(6k + 1)$.  For $\gamma = \beta_0 + n$ we will let

\begin{center}

$I_{\Log, \gamma}  = I_{\Log, \beta_0}$

$I_{\Rel, \gamma} = I_{\Rel, \beta_0}$

$I_{\Conn, \gamma} = I_{\Conn, \beta_0}$

$\Gamma_{\gamma} = \Gamma_{\beta_0}$

$\Sequ_{\gamma} = \Sequ_{\beta_0}$

\end{center}

\noindent  We know by assumption that (iii) holds at $\beta_0$, so that in particular $\Gamma_{\beta_0}$ is consistent.  Let $\{(Y_m, Z_m)\}_{m \in \omega}$ be an enumeration of all pairs $(Y, Z)$ with $Y, Z \subseteq I_{\Rel, \beta_0}$ finite and $\X_{\beta_0, Y, Z} \neq \emptyset$ and such that $(Y_m, Z_m) = (Y_{m'}, Z_{m'})$ for infinitely many $m'\in \omega$.  Select $\sigma_m \in \X_{\beta_0, Y_m, Z_m}$ for each $m \in \omega$.

We proceed by induction on $n \geq 1$ and we let $m = n - 1$.  For each $\alpha \in I_{\Rel, \beta_0}$ we let

\begin{center}

$B_{\alpha}^{\beta_0 + m + 1} = \begin{cases} B_{\alpha}^{\beta_0 + m}$ if $\sigma_m(\alpha) = 0\\  B_{\alpha}^{\beta_0 + m} \cup \{\beta_0 + m\}$ if $\sigma_m(\alpha) = 1    \end{cases}$

\end{center}

\noindent Let $\Bo^{\beta_0 + m + 1} = \{B_{\alpha}^{\beta_0 + m + 1}\}_{\alpha \in I_{\Rel, \beta_0 + m}}$.  By Lemma \ref{addingpoints}, we have $\Bo^{\beta_0 + m} \models \Gamma_{\beta_0} = \Gamma_{\beta_0 + m}$ for all $m \in \omega$ by induction on $m \in \omega$.  Thus (iii) holds at each $\beta + n$.  All other induction requirements (i), (ii), (iv), (v), (vi) are quite obvious.

We observe that the above construction has guaranteed that $$([\beta_0, \beta_0 + \omega), \{[\beta_0, \beta_0 + \omega)\cap B_{\alpha}^{\beta_0 + \omega}\}_{\alpha \in I_{\Rel, \beta_0 + \omega}})$$ is a Boolean saturated model of $\Gamma_{\beta_0 + \omega} = \Gamma_{\beta_0}$.

\noindent \textbf{Case F: $\gamma = t + \omega 6k + n, n \geq 1$.}  Let $\beta = t + \omega 6k$.  For all such $\gamma$ we let 

\begin{center}

$I_{\Log, \gamma} = I_{\Log, \beta}$

$\Gamma_{\gamma} = \Gamma_{\beta}$

$I_{\Rel, \gamma} = I_{\Rel, \beta}$

\end{center}

\noindent \textbf{Subcase F.1:  $\gamma = t + 1$ with $t\in T$ and there exist nonempty sets $L_0, L_1 \subseteq I_{\Rel, t}$ such that $t = (\bigcup_{\alpha \in L_0} B_{\alpha}^t) \sqcup (\bigcup_{\alpha \in L_1} B_{\alpha}^t)$, $\bigcup_{\alpha \in L_0} B_{\alpha}^t = S_{t, 0}$ and $\bigcup_{\alpha \in L_1} B_{\alpha}^t = S_{t, 1}$}.  Fix such $L_0, L_1 \subseteq I_{\Rel, t}$.  We first notice that $(t , \Bo^t)$ is a Boolean saturated model of $\Gamma_t$, since $(\omega, \{\omega \cap B_{\alpha}^t\}_{\alpha \in I_{\Rel, t}})$ is a Boolean saturated model.

Next we notice that $\bigcup_{\alpha \in L_0} b_{\alpha}^t$ is disjoint from $\bigcup_{\alpha \in L_1} b_{\alpha}^t$.  Were disjointness to fail, we would have $\alpha_0 \in L_0$ and $\alpha_1 \in L_1$ with $b_{\alpha_0}^t \cap b_{\alpha_1}^t \neq \emptyset$, but since $(t , \Bo^t)$ is Boolean saturated we would have $B_{\alpha_0}^t \cap B_{\alpha_1}^t \neq \emptyset$.

Also, $(\bigcup_{\alpha \in L_0}b_{\alpha}^t) \cup (\bigcup_{\alpha \in L_1}b_{\alpha}^t)$ is dense in $\X_t$.  This follows from the fact that $\Coor_t(t)$ is dense in $\X_t$ (by Lemma \ref{coord}) and $\Coor_t(t) \subseteq (\bigcup_{\alpha \in L_0}b_{\alpha}^t) \cup (\bigcup_{\alpha \in L_1}b_{\alpha}^t)$.

As $\X_t$ is connected (Lemma \ref{compactconnected}), by Lemma \ref{limitpoint} we select $\sigma \in \X_t$ which is a limit point of $\bigcup_{\alpha \in L_0}b_{\alpha}^t$ and also a limit point of $\bigcup_{\alpha \in L_1}b_{\alpha}^t$.  Let $L_2 = \{\alpha \in I_{\Rel, t} \mid \sigma \in b_{\alpha}^t\}$.  Notice that $L_2$ is nonempty, since $\X_t$ models some statements of type (c) (by Case C of our induction).  Let $L_2 = \{\alpha_q\}_{q \in \omega}$ be an enumeration (we allow repetition of indices for the same element of $L_2$).  For each $q \in \omega$ we select $\alpha_{0, q} \in L_0$ and $\alpha_{1, q} \in L_1$ such that $b_{\alpha_{0, q}}^t \cap \bigcap_{i = 0}^q b_{\alpha_i}^t \neq \emptyset$ and $b_{\alpha_{1, q}}^t \cap \bigcap_{i = 0}^q b_{\alpha_i}^t \neq \emptyset$.  Let $\{\epsilon_q\}_{q \in \omega}$ be a strictly increasing sequence such that $\sup_{q \in \omega}\epsilon_q = t$.  Pick $\omega <\beta_0 < t$ such that $\beta_0$ is of form $\beta_0 = t_0 + \omega(6k_0+1)$ and $\beta_0 > \alpha_0, \alpha_{0, 0}, \alpha_{1, 0}, \epsilon_0$.  Assuming $\beta_q$ has been selected we select $\beta_{q+1}$ which is of form $\beta_{q+1} = t_{q+1} + \omega(6k_{q + 1} + 1)$ and $\beta_{q + 1} > \beta_q, \alpha_0, \ldots, \alpha_q, \alpha_{0, q}, \alpha_{1, q}, \epsilon_q$.

We have by construction that $\{\beta_q\}_{q \in \omega}$ is a strictly increasing sequence with $\sup_{q \in \omega}\beta_q = t$.  Since $([\beta_q, \beta_q + \omega), \{[\beta_q, \beta_q + \omega) \cap B_{\alpha}^{\beta_q  + \omega}\}_{\alpha \in I_{\beta_q + \omega}})$ is Boolean saturated for each $q\in \omega$ we select $s_{t, 0, q}, s_{t, 1, q}\in [\beta_q, \beta_q + \omega)$ such that

\begin{center}
$\Coor_{\beta_q + \omega}(s_{t, 0, q}) \in b_{\alpha_{0, q}}^{\beta_q + \omega}\cap \bigcap_{i = 0}^q b_{\alpha_i}^{\beta_q + \omega}$

$\Coor_{\beta_q + \omega}(s_{t, 1, n}) \in b_{\alpha_{1, q}}^{\beta_q + \omega} \cap \bigcap_{i = 0}^q b_{\alpha_i}^{\beta_q + \omega}$
\end{center}

For all $\alpha \in I_{\Rel, \gamma} = I_{\Rel, \beta}$ we let

\begin{center}

$B_{\alpha}^{\gamma} = \begin{cases}B_{\alpha}^t$ if $\sigma(\alpha) = 0\\ B_{\alpha}^t \cup \{t\}$ if $\sigma(\alpha) = 1  \end{cases}$

\end{center}

\noindent We know $(\gamma, \{B_{\alpha}^{\gamma}\}_{\alpha \in I_{\Rel, \gamma}})$ is a model of $\Gamma_{\gamma}$ by Lemma \ref{addingpoints}.  Let $I_{\Conn, \gamma} = I_{\Conn, t} \cup \{t\}$ and $\Sequ_{\gamma} = \Sequ_t \cup \{(\{s_{t, 0, q}\}_{q \in \omega}, \{s_{t, 1, q}\}_{q \in \omega})\}$. That (i) - (vi) hold at $\gamma$ is a straightforward to check.  For example, for any $\alpha \in I_{\Rel, t+ 1}$, if $t \in B_{\alpha}^{t + 1}$ then the sequence $\{s_{t, 0, q}\}_{q\in \omega}$ and the sequence $\{s_{t, 1, q}\}_{q\in \omega}$ are eventually in $B_{\alpha}^{t + 1}$.

\noindent \textbf{Subcase F.2:  $\gamma = t + \omega 6k +1, \gamma \notin \omega$ and Subcase F.1 fails.}  Pick $\sigma \in \X_t$ and for all $\alpha \in I_{\Rel, \gamma}$ we let

\begin{center}

$B_{\alpha}^{\gamma} = \begin{cases} B_{\alpha}^{\beta_0}$ if $\sigma(\alpha) = 0 \\ B_{\alpha}^{\beta_0} \cup \{\beta_0\}$ if $\sigma(\alpha) = 1  \end{cases}$

\end{center}

\noindent By Lemma \ref{addingpoints} we see that $(\gamma, \{B_{\alpha}^{\gamma}\}_{\alpha \in I_{\Rel, \gamma}})$ is a model of $\Gamma_{\gamma}$, and the other inductive properties are easily seen to still hold.

\noindent \textbf{Subcase F.3: $\gamma = t + \omega 6k +n, \gamma \notin \omega, n \geq 2$.}  We will let $I_{\Conn, \gamma} = I_{\Conn, \gamma - 1}$.  Recall the function $f: R \rightarrow \bigcup_{\delta < \aleph_1}\{0, 1\}^{\delta}$ chosen at the onset of our construction.  If $J$ is a well-ordered set (for example, a subset of an ordinal) we let $\ord(J)$ denote the unique ordinal which is order isomorphic to $J$.  Given any function $g \in \bigcup_{\delta < \aleph_1}\{0, 1\}^{\delta}$ we let $\dom(g)$ denote the domain.  We will look at two possibilities.

Suppose $\dom(f(\gamma)) \leq \ord(I_{\Rel, \gamma})$.  Let $\iota:  \dom(f(\gamma)) \rightarrow \ord(I_{\Rel, \gamma})$ be the unique order embedding of $\dom(f(\gamma))$ as an initial interval of $\ord(I_{\Rel, \gamma})$.  If there exists $\sigma \in \X_{\gamma - 1} = \X_{\beta_0}$ such that $\sigma(\iota(\epsilon)) = (f(\gamma))(\epsilon)$ for all $\epsilon \in \dom(f(\gamma))$, then we select such a $\sigma \in \X_{\beta_0}$ and for all $\alpha \in I_{\Rel, \beta_0} = I_{\Rel, \gamma - 1} = I_{\Rel, \gamma}$ we let

\begin{center}

$B_{\alpha}^{\gamma} = \begin{cases}B_{\alpha}^{\gamma - 1}$ if $\sigma(\alpha) = 0\\ B_{\alpha}^{\gamma - 1} \cup \{\gamma - 1\}$ if $\sigma(\alpha) = 1 \end{cases}$

\end{center}

\noindent If either $\dom(f(\gamma)) \leq \ord(I_{\Rel, \gamma})$ or there does not exist such a $\sigma \in \X_{\gamma}$, we select an arbitrary $\sigma \in \X_{\gamma}$ and again for all $\alpha \in I_{\Rel, \gamma}$ we let

\begin{center}

$B_{\alpha}^{\gamma} = \begin{cases}B_{\alpha}^{\gamma - 1}$ if $\sigma(\alpha) = 0\\ B_{\alpha}^{\gamma - 1} \cup \{\gamma - 1\}$ if $\sigma(\alpha) = 1 \end{cases}$

\end{center}

We have now considered all ordinals $\gamma < \aleph_1$.  Let $I_{\Rel} = \bigcup_{\beta < \aleph_1} I_{\Rel, \beta}$.  For each $\alpha \in I_{\Rel}$ we let $B_{\alpha} = \bigcup_{\alpha \leq \beta } B_{\alpha}^{\beta}$.  Let $\Bo = \{B_{\alpha}\}_{\alpha \in I_{\Rel}}$.  Let $I_{\Log} = \bigcup_{\beta < \aleph_1} I_{\Log, \beta}$ and $\Gamma = \bigcup_{\beta < \aleph_1} \Gamma_{\beta}$.  Let $\Sequ = \bigcup_{\beta < \aleph_1} \Sequ_{\beta}$.  Let $\tau$ denote the topology $\tau(\Bo)$ on $\aleph_1$.  We check the properties that ought to hold for $(\aleph_1, \tau)$ (including that $\Bo$ is a basis for $\tau$) in the following section.
\end{section}

\begin{section}{Verification of Theorem \ref{main}}\label{verification}

In this section we finish the proof of Theorem \ref{main} by verifying the various claims about the space $(\aleph_1, \tau)$.  

\begin{remark}\label{modelindeed}  It is straightforward to check that $(\aleph_1, \Bo) \models \Gamma$ and also that $(\omega, \{\omega \cap B_{\alpha}\}_{\alpha \in I_{\Rel}})$ is a (Boolean saturated) model of $\Gamma$, arguing precisely as in Case L of our construction.  That $\Gamma$ satisfies $\dagger$ is also clear by construction.
\end{remark}

\begin{lemma}\label{basisindeed}  $\Bo$ is a basis for the topology $\tau = \tau(\Bo)$ on $\aleph_1$.  
\end{lemma}

\begin{proof}  Since $(\aleph_1, \Bo) \models \Gamma$ and by Case C of our construction, it is clear that $\bigcup \Bo = \aleph_1$. Suppose that $x \in B_{\alpha_0} \cap B_{\alpha_1}$.  Select $\beta_0 < \aleph_1$ which is of form $\beta_0 = t + \omega(6k + 3)$ large enough that $x \in \beta_0$ and $\alpha_0, \alpha_1 \in  I_{\Rel, \beta_0}$.  We have $B_{\alpha_0}^{\beta_0} = \beta_0 \cap B_{\alpha_0}$ and $B_{\alpha_1}^{\beta_0} = \beta_0 \cap B_{\alpha_1}$, so in particular $x\in  B_{\alpha_0}^{\beta_0} \cap B_{\alpha_1}^{\beta_0}$.  By the treatment of Case B there exists $\beta_0 \leq \alpha_2 < \beta_0 + \omega$ such that $x \in B_{\alpha_2}^{\alpha_2} \subseteq B_{\alpha_0}^{\alpha_2} \cap B_{\alpha_1}^{\alpha_2}$ and $\Theta_{\alpha_2}\in \Gamma_{\alpha_2 + 1}$ with $\Theta_{\alpha_2} \equiv [\B_{\alpha_0} \cap \B_{\alpha_1} \supseteq \B_{\alpha_2}]\wedge [\B_{\alpha_2} \neq \emptyset]$.  Since $(\aleph_1, \Bo)$ models $\Gamma$, we have $x\in B_{\alpha_2} \subseteq B_{\alpha_0}\cap B_{\alpha_1}$.
\end{proof}

\begin{lemma}\label{regularindeed}(Property (1))  The space $(\aleph_1, \tau)$ is regular.
\end{lemma}

\begin{proof}  If $x, y \in \aleph_1$ are distinct then we can select $\beta_0 < \aleph_1$ which is of form $\beta_0 = t + \omega(6k + 2)$ with $x, y < \beta_0$.  By the treatment of Case A there exist $\beta_0 \leq \alpha_0 < \alpha_1 <\beta_0 + \omega$ such that $B_{\alpha_0}^{\alpha_1} \cap B_{\alpha_1}^{\alpha_1} = \emptyset$, $x\in B_{\alpha_0}^{\alpha_1}$ and $y\in B_{\alpha_1}^{\alpha_1}$, and accompanying $\Theta_{\alpha_1}$ of type (a).  That $x \in B_{\alpha_0}$, $y \in B_{\alpha_1}$ and $B_{\alpha_0} \cap B_{\alpha_1} = \emptyset$ is immediate.  Thus the space is Hausdorff, and regularity is argued along similar lines, using Case D and the fact that $\Bo$ is a basis for $\tau$.
\end{proof}

\begin{lemma}\label{separableindeed}(Property (2))  The space $(\aleph_1, \tau)$ is separable.
\end{lemma}

\begin{proof}  Since $(\omega, \{\omega \cap B_{\alpha}\}_{\alpha \in I_{\Rel}}) \models \Gamma$, we know more particularly that $B_{\alpha}\cap \omega \neq \emptyset$ for all $\alpha \in I_{\Rel}$, which demonstrates that $\omega$ is dense in $(\aleph_1, \tau)$.

\end{proof}

Towards property (3) we prove the following.

\begin{lemma} \label{opensetsaccumulate}  Suppose that $\emptyset \neq L \subseteq I_{\Rel}$.  The set 

\begin{center}
$J = \{\beta \in \aleph_1 \mid \emptyset \neq \bigcup_{\alpha \in I_{\Rel, \beta} \cap L} B_{\alpha}^{\beta} = \beta \cap \bigcup_{\alpha \in L} B_{\alpha}\}$
\end{center}

\noindent is club in $\aleph_1$.

\end{lemma}

\begin{proof}  To see that $J$ is unbounded we let $\beta_0 < \aleph_1$ be given, without loss of generality such that $I_{\Rel, \beta_0} \cap L \neq \emptyset$.  Supposing that $\beta_n$ has already been defined we select $\aleph_1 > \beta_{n + 1} > \beta_n$ large enough that $$(\beta_n \cap \bigcup_{\alpha \in L} B_{\alpha}) \setminus \bigcup_{\alpha \in I_{\Rel, \beta_n}\cap L} B_{\alpha}^{\beta_n} \subseteq \bigcup_{\alpha \in I_{\Rel, \beta_{n+ 1}} \cap L} B_{\alpha}^{\beta_{n+1}}$$ (this is possible since the set on the left is countable).  Letting $\beta = \sup_{n\in \omega} \beta_n$ it is easy to see that $\beta \cap \bigcup_{\alpha \in L} B_{\alpha} \subseteq  \bigcup_{\alpha \in I_{\Rel, \beta_{n+ 1}} \cap L} B_{\alpha}^{\beta}$, and the reverse inclusion holds necessarily, and clearly the set $\beta \cap \bigcup_{\alpha \in L} B_{\alpha}$ is nonempty as $I_{\Rel, \beta} \cap L \supseteq I_{\Rel, \beta_0} \cap L \neq \emptyset$.  Thus $\beta \in J$.

To see that $J$ is closed we let $\{\beta_n\}_{n\in \omega}$ be an increasing sequence in $J$ and $\beta = \sup_{n \in \omega} \beta_n$.  Then

\begin{center}
$\bigcup_{\alpha \in I_{\Rel, \beta}\cap L} B_{\alpha}^{\beta} \supseteq \bigcup_{n\in \omega} (\bigcup_{\alpha \in I_{\Rel, \beta_n} \cap L} B_{\alpha}^{\beta_n})$

$= \bigcup_{n \in \omega} (\beta_n \cap (\bigcup_{\alpha \in L} B_{\alpha}))$

$= \beta \cap (\bigcup_{\alpha \in L} B_{\alpha})$

$\supseteq \bigcup_{\alpha \in I_{\Rel, \beta}\cap L} B_{\alpha}^{\beta}$

\end{center}

\noindent and as $\bigcup_{\alpha \in I_{\Rel, \beta_0} \cap L} \neq \emptyset$, the set $\bigcup_{\alpha \in I_{\Rel, \beta}\cap L} B_{\alpha}^{\beta} = \beta \cap (\bigcup_{\alpha \in L} B_{\alpha})$ is nonempty.  Thus $\beta \in J$.

\end{proof}

\begin{lemma} \label{connectedindeed}(Property (3))  The space $(\aleph_1, \tau)$ is connected.
\end{lemma}

\begin{proof}  Suppose for contradiction that $\aleph_1 = V_0 \sqcup V_1$ is a nontrivial disconnection.  Then we can select nonempty $L_0, L_1 \subseteq I_{\Rel}$ such that $V_0 = \bigcup_{\alpha \in L_0} B_{\alpha}$ and $V_1 = \bigcup_{\alpha \in L_1} B_{\alpha}$.  For $i \in \{0, 1\}$ we let

\begin{center}
$J_i = \{\beta \in \aleph_1 \mid \emptyset \neq \bigcup_{\alpha \in I_{\Rel, \beta} \cap L_i} B_{\alpha}^{\beta} = \beta \cap \bigcup_{\alpha \in L_i} B_{\alpha}\}$
\end{center}

\noindent and by Lemma \ref{opensetsaccumulate} each of $J_0, J_1$ is club in $\aleph_1$.  By our selection of sequence $\{(S_{\alpha, 0}, S_{\alpha, 1})\}_{\alpha \in T}$ we know that the set

\begin{center}
$J_2 = \{t \in T \mid [t \cap V_0 = S_{t, 0}] \wedge [t \cap V_1 = S_{t, 1}]\}$
\end{center}

\noindent is stationary in $\aleph_1$.  Then we can select $t \in J_0 \cap J_1 \cap J_2$, and $t + 1$ is considered in Subcase F.1.  Thus we have $t \in I_{\Conn, t+1}$ and sequences $\{s_{t, 0, q}\}_{q \in \omega}$ and $\{s_{t, 1, q}\}_{q \in \omega}$ such that $\{s_{t, i, q}\}_{q\in \omega} \subseteq S_{t, i}$ for $i\in \{0, 1\}$.  Either $t \in V_0$ or $t \in V_1$, so that either $\emptyset \neq S_{t, 0} \cap V_1 \subseteq V_0 \cap V_1$ or $\emptyset \neq V_0 \cap S_{t, 1} \subseteq V_0 \cap V_1$, and in either case we obtain a contradiction.  Thus $(\aleph_1, \tau)$ is connected.

\end{proof}

Property (4), that the space is of cardinality $\aleph_1$, is quite obvious.

\begin{lemma} \label{separatesindeed}(Property (5))  For every nonempty open $O$ in the space $(\aleph_1, \tau)$ the subspace $\aleph_1 \setminus O$ is totally separated.
\end{lemma}

\begin{proof}  Clearly it suffices to check the claim in the case where $O$ is a basis element.  Let $\alpha \in I_{\Rel}$ be given, as well as distinct $x, y \in \aleph_1 \setminus B_{\alpha}$.  Select $\beta_0 < \aleph_1$ of form $\beta_0 = t + \omega(6k + 4)$ which is large enough that $x, y \in \beta_0$ and $\alpha \in I_{\Rel, \beta_0}$.  By the treatment of Case C there exist $\alpha_0, \alpha_1 \in [\beta_0, \beta_0 + \omega)$ such that $x \in B_{\alpha_0}^{\alpha_1}$, $y\in B_{\alpha_1}^{\alpha_1}$, and $B_{\alpha_0}^{\alpha_1} \cap B_{\alpha_1}^{\alpha_1} \subseteq B_{\alpha}^{\alpha_1}$, and accompanying $\Theta_{\alpha_1}$ of type (c).  Thus $x \in B_{\alpha_0}$, $y\in B_{\alpha_1}$, $B_{\alpha_0} \cap B_{\alpha_1} \subseteq B_{\alpha}$, and $B_{\alpha} \cup B_{\alpha_0} \cup B_{\alpha_1} = \aleph_1$, as $(\aleph_1, \Bo) \models \Gamma$.
\end{proof}

\begin{lemma} \label{bigopensindeed}  (Property(6)) Each nonempty open subset in $(\aleph_1, \tau)$ is uncountable.
\end{lemma}

\begin{proof}  It suffices to prove that each basis element is uncountable.  Given $\alpha \in I_{\Rel}$ we notice that for each $\beta_0 < \aleph_1$ of form $\beta_0 = t + \omega(6k + 1)$ large enough that $\alpha \in I_{\Rel, \beta_0}$ we have $B_{\alpha}^{\beta_0 + \omega} \cap [\beta_0, \beta_0 + \omega) \neq \emptyset$ since $$([\beta_0, \beta_0 + \omega) , \{[\beta_0, \beta_0 + \omega) \cap B_{\alpha}^{\beta_0 + \omega}\}_{\alpha \in I_{\Rel, \beta_0 + \omega}})$$ is a (Boolean saturated) model of the set of sentences $\Gamma_{\beta_0}$ which includes a sentence requiring $\B_{\alpha} \neq \emptyset$.  That $B_{\alpha}$ is uncountable follows immediately.
\end{proof}

\begin{lemma}  \label{uncountablebasisindeed}(Property (7)) $|\Bo| = \aleph_1$.
\end{lemma}

\begin{proof}  That $|\Bo| \leq \aleph_1$ is obvious since the elements are indexed by $I_{\Rel} \subseteq \aleph_1$.  Instead of directly verifying that $\Bo$ is uncountable, we will use the forthcoming property (8).  Were $\Bo$ countable, $(\aleph_1, \tau)$ would be Lindel\"of, and by (8) we would have $(\aleph_1, \tau)$ compact, which together with (1), (3), (4) and (5) gives a contradiction (see the proof of Proposition \ref{badprop}).
\end{proof}

Towards properties (8) and (9) we give the following.

\begin{lemma} \label{thickness}  For each $\beta < \aleph_1$ and $\sigma \in \X_{\beta}$ there exists $\beta_0 > \beta$ and $x \in \beta_0$ such that $\Coor_{\beta_0}(x) \upharpoonright I_{\Rel, \beta} = \sigma$.
\end{lemma}

\begin{proof}  Let $\delta = \ord(I_{\Rel, \beta})$ and $\iota: \delta \rightarrow I_{\Rel, \beta}$ be the unique order isomorphism.  Let $g \in \{0, 1\}^{\delta}$ be given by $g(\epsilon) = \sigma(\iota^{-1}(\epsilon))$.  Select $\beta_0 \in R$ such that $\beta_0 > \beta$ and $f(\beta_0) = g$.  By Subcase F.3 we have $\Coor_{\beta_0}(\beta_0 - 1) \upharpoonright I_{\Rel, \beta} = \sigma$.
\end{proof}

\begin{lemma}\label{nicecoverindeed}(Property (8))  Every countable cover of $\aleph_1$ by elements of $\Bo$ has a finite subcover.
\end{lemma}

\begin{proof}  Let $\{\alpha_n\}_{n\in \omega} \subseteq I_{\Rel}$ be such that $\bigcup_{n\in \omega} B_{\alpha_n} = \aleph_1$.  Select $\beta < \aleph_1$ large enough that $\alpha_n \in I_{\Rel, \beta}$ for all $n \in \omega$.  We claim that $\{b_{\alpha_n}^{\beta}\}_{n\in \omega}$ is a cover of $\X_{\beta}$.  To see this, we suppose otherwise and let $\sigma \in \X_{\beta} \setminus \bigcup_{n\in \omega} b_{\alpha_n}^{\beta}$.  By Lemma \ref{thickness} we select $\beta_0 > \beta$ and $x\in \beta_0$ such that $\Coor_{\beta_0}(x) \upharpoonright I_{\Rel, \beta} = \sigma$.  Clearly $x\in \beta_0 \setminus \bigcup_{n\in \omega} B_{\alpha_n}^{\beta_0}$, and so $x\in \aleph_1 \setminus \bigcup_{n\in \omega} B_{\alpha_n}$, a contradiction.

Thus $\{b_{\alpha_n}^{\beta}\}_{n\in \omega}$ is a cover of $\X_{\beta}$, and as $(\X_{\beta}, \tau(\{b_{\alpha}^{\beta}\}_{\alpha \in I_{\Rel, \beta}}))$ is compact (by Lemma \ref{compactconnected}) we may select a finite subcover $\{b_{\alpha_{k_0}}^{\beta}, \ldots, b_{\alpha_{k_n}}^{\beta}\}$.  Given arbitrary $x \in \aleph_1$ we select $\beta_0 \geq \beta, x$ and notice that $\Coor_{\beta_0}(x) \upharpoonright I_{\Rel, \beta} \in \bigcup_{i = 0}^n b_{\alpha_{k_i}}^{\beta}$.  Thus $\{B_{\alpha_{k_0}}, \ldots, B_{\alpha_{k_n}}\}$ is a subcover.
\end{proof}

\begin{lemma}\label{nonemptyintersectionindeed}(Property (9))  If $\{B_n\}_{n\in \omega}$ is a nesting decreasing sequence of elements of $\Bo$ then $\bigcap_{n\in \omega} B_n \neq \emptyset$, and so $(\aleph_1, \tau)$ is strongly Choquet.
\end{lemma}

\begin{proof}  That $(\aleph_1, \tau)$ is strongly Choquet will follow immediately from the first claim (see Lemma \ref{stronglyChoquetfrombasis}).  Supposing that $\{\alpha_n\}_{n\in \omega}$ is a sequence in $I_{\Rel}$ such that $B_{\alpha_0} \supseteq B_{\alpha_1} \supseteq \cdots$ we let $\beta_0 < \aleph_1$ be large enough that $\alpha_n \in I_{\Rel, \beta_0}$ for all $n\in \omega$ and of form $\beta_0 = t + \omega(6k + 1)$.  By Case E we have $([\beta_0, \beta_0 + \omega), \{[\beta_0, \beta_0 + \omega) \cap B_{\alpha}^{\beta_0 + \omega}\}_{\alpha \in I_{\Rel, \beta_0 + \omega}})$ is a Boolean saturated model of $\Gamma_{\beta_0 + \omega}$.

Notice that $b_{\alpha_0}^{\beta_0 + \omega} \supseteq b_{\alpha_1}^{\beta_0 + \omega} \supseteq \cdots$, for if there existed $\sigma \in b_{\alpha_{n + 1}}^{\beta_0 + \omega} \setminus b_{\alpha_n}^{\beta_0 + \omega}$ we would have, by Boolean saturation, some $x\in [\beta_0, \beta_0 + \omega)$ with $x \in B_{\alpha_{n+1}}^{\beta_0 + \omega} \setminus B_{\alpha_n}^{\beta_0 + \omega}$, so $x \in B_{\alpha_{n+1}} \setminus B_{\alpha_n}$, contrary to assumption.

By Lemma \ref{compactsubspace} we know that $\X_{\beta_0 + \omega}$ is compact as a subset of $\{0, 1\}^{I_{\Rel, \beta _0+ \omega}}$.  Each $b_{\alpha_n}^{\beta_0 + \omega}$ is the intersection of $\X_{\beta_0 + \omega}$ with a closed subset of $\{0, 1\}^{\beta_0 + \omega}$.  As $(\X_{\beta_0 + \omega}, \{b_{\alpha}^{\beta_0 + \omega}\}_{\alpha \in I_{\Rel, \beta_0 + \omega}}) \models \Gamma_{\beta_0 + \omega}$ (Lemma \ref{nicemodel}), we know more particularly that $b_{\alpha_n}^{\beta_0 + \omega} \neq \emptyset$.  Thus $\bigcap_{n \in \omega} b_{\alpha_n}^{\beta_0 + \omega} \neq \emptyset$ by compactness.  Selecting $\sigma \in \bigcap_{n \in \omega} b_{\alpha_n}^{\beta_0 + \omega}$, we pick by Lemma \ref{thickness} an element $\beta_1 > \beta_0 + \omega$ and $x \in \beta_1$ such that $\Coor_{\beta_1}(x) \upharpoonright I_{\Rel, \beta_0 + \omega} = \sigma$.  Clearly $x \in \bigcap_{n \in \omega} B_{\alpha_n}^{\beta_1} \subseteq  \bigcap_{n \in \omega} B_{\alpha_n}$, so we are done.

\end{proof}

\end{section}

\section*{Acknowledgement}

The author is enormously grateful to Jan van Mill and Adam Ostaszewski for feedback on a draft of this article and also to David Lipham for sharing his knowledge of the literature on widely connected spaces.


\begin{thebibliography}{A}

\bibitem{CiWo} K. C. Ciesielski, J. Wojciechowski, \emph{Cardinality of regular spaces admitting only constant continuous functions}, Topology Proc. 47 (2016), 313-329.

\bibitem{Gru} G. Gruenhage, \emph{Spaces in which the nondegenerate connected sets are the cofinite sets}, Proc. Amer. Math. Soc. 122 (1994), 911-924.

\bibitem{Jec} T. Jech, Set Theory: The Third Millenium Edition, Revised and Expanded, Springer-Verlag, 2006.

\bibitem{Jen} R. Jensen, \emph{The fine structure of the constructible hierarchy}, Ann. Math. Logic 4 (1972), 229-308.

\bibitem{Lip} D. Lipham, \emph{Widely-connected sets in the bucket-handle continuum}, Fund. Math. 240 (2018), 161-174.

\bibitem{Ke} A. Kechris, Classical Descriptive Set Theory, Springer-Verlag, 1995.

\bibitem{KnKu} B. Knaster, C. Kuratowski, \emph{Sur les ensembles connexes}, Fund. Math. 2 (1921), 206-255.

\bibitem{StSe} L. A. Steen, J. A. Seebach, Counterexamples in Topology, 2nd Edition, Springer-Verlag,1978.

\bibitem{Sw} P. M. Swingle, \emph{Two types of connected sets}, Bull. Amer. Math. Soc. 37 (1931), 254-258.



\end{thebibliography}
\end{document}